\documentclass[a4paper, 12pt]{amsart}
\usepackage[T1]{fontenc}
\usepackage{amssymb,amsthm,amsmath}
\usepackage{xspace}
\usepackage{enumerate}

\newcommand{\R}{\mathbb{R}}

\newcommand{\N}{\mathbb{N}}

\newcommand{\cA}{\mathcal{A}}

\newcommand{\cE}{\mathcal{E}}
\newcommand{\cF}{\mathcal{F}}

\newcommand{\cN}{\mathcal{N}}

\newtheorem{theorem}{Theorem}[section]
\newtheorem{lemma}[theorem]{Lemma}
\newtheorem{proposition}[theorem]{Proposition}

\theoremstyle{remark}

\newtheorem{example}[theorem]{Example}
\newtheorem{definition}[theorem]{Definition}

\begin{document}

\title[]
{Stability of travelling waves in stochastic bistable reaction-diffusion equations}
\author{Wilhelm Stannat}
\address{%
Institut f\"ur Mathematik\\
Technische Universit\"at Berlin \\
Stra{\ss}e des 17. Juni 136\\
D-10623 Berlin\\
and\\
Bernstein Center for Computational Neuroscience\\
Philippstr. 13\\
D-10115 Berlin\\
Germany}
\email{stannat@math.tu-berlin.de}
\date{Berlin, April 14, 2014}

\begin{abstract}
We prove stability of travelling waves for stochastic bistable reaction-diffusion equations with both additive and multiplicative noise, using a 
variational approach based on functional inequalities. Our analysis yields explicit estimates on the rate of stability that can be shown in 
special examples to be optimal. 
\end{abstract}


\maketitle

\section{Introduction} 

\noindent 
The purpose of this paper is to generalize the main results of \cite{St} on the stability of travelling waves  
in Nagumo equation with multiplicative noise to general bistable reaction diffusion equations with noise. To this 
end let us first consider the deterministic reaction-diffusion equation 
\begin{equation} 
\label{RDE} 
\partial_t v (t,x) = \nu v_{xx}(t,x) + b f(v(t,x)) \, , \quad v(t,x) = v_0 (x)
\end{equation} 
for $(t,x)\in\R_+\times\R$. Here, $f: \R\to\R$ is a continuously differentiable function satisfying  
\begin{equation}
\tag{\bf A1}
\begin{aligned}
& f (0)  = f(a) = f (1) = 0  \quad \text{ for some }  a \in (0,1) \\
& f(x) < 0 \quad \text{ for }  x \in (0,a) \, , f(x) > 0 \text{ for } x \in (a,1) \\
& f^\prime (0) < 0 , f^\prime (a) > 0, f^\prime (1) < 0\, .  
\end{aligned}
\end{equation}  
Theorem 12 in \cite{HR} implies for $\nu$, $b > 0$ the existence of a travelling wave connecting the stable fixed points 
$0$ and $1$ of the reaction term, i.e., a monotone increasing $C^2$ function $\hat{v}$ satisfying 
$$ 
c \hat{v}_x = \nu \hat{v}_{xx} + bf (\hat{v}) 
$$ 
for some wavespeed $c\in\R$ and boundary conditions $\hat{v} (- \infty) = 0$, $\hat{v} (+ \infty) = 1$. It follows that 
$\hat{v}(t) := \hat{v} (\cdot + ct)$ and all its spatial 
translates $\hat{v} (\cdot + x_0 + ct)$ are solutions of \eqref{RDE}. A particular example is the Nagumo equation with 
$f(v) = v (1-v)(v -a)$ where the travelling wave is explicitely given by $\hat{v}(x) = \left( 1 + e^{- \sqrt{\frac{b}{2v}} x}\right)^{-1}$. 

\medskip 
\noindent 
It is known that the wave speed $c$ and the integral $\int_0^1 f(v)\, dv \ge 0$ have the same sign and that in particular $c=0$ if and only if 
$\int_0^1 f(v)\, dv = 0$. To simplify the presentation of our results we will therefore assume from now on that 
\begin{equation}
\tag{\bf A2}
\int_0^1 f(v)\, dv \ge  0 
\end{equation} 
hence that the wave speed $c$ is nonnegative. 

\medskip 
\noindent 
So far the assumptions on the reaction term $f$ are classical. The existing results in the literature on the stability of the travelling 
wave can be divided up into results based on maximum principle and comparison techniques, see in particular \cite{FMcL} 
for a stability result w.r.t. initial conditions $v_0$ satisying $0\le v_0\le 1$, $\liminf_{x\to-\infty} v_0 (x) < a$ and 
$\limsup_{x\to\infty} v_0 (x) > a$, and results w.r.t. $L^2$- or $H^{1,2}$-norms, based on spectral information on the linearization of \eqref{RDE} 
along the travelling wave $\hat{v}$ (see, e.g. \cite{Henry, OR}). Whereas the first approach is not appropriate for stochastic  perturbations, unless 
the noise terms would satisfy unnatural monotonicity conditions, the second approach can be in principle generalized to the stochastic case. 
However, in order to do this, the existing spectral information on the linearization of \eqref{RDE} has to be considerably refined. Abstract 
perturbation results on the spectral gap below the eigenvalue corresponding to the travelling wave cannot be easily generalized to the stochastic case. 
We will therefore use functional inequalities to derive Lyapunov stability of the travelling wave in the space $L^2 (\R)$. To be more precise, we will show in 
Theorem \ref{th1} under the following additional assumptions on the reaction term 

\begin{equation} 
\tag{\bf A3} 
\begin{aligned} 
& \exists v_\ast\in (a, 1) \mbox{ such that } f'' (v) > 0 \, (\mbox{  resp. } < 0) \\ 
& \mbox{ on } [0 , v_\ast) (\mbox{ resp. }(v_\ast , 1]) 
\end{aligned} 
\end{equation} 
saying that $f$ is strictly convex on $[0, v_\ast )$ and strictly concave on $(v_\ast , 1]$, that the $L^2$-norm is a Lyapunov 
function restricted to the orthogonal complement of $\hat{v}_x$. As a consequence of this phase-space stability, the stochastic case will 
become much easier to investigate. Our assumptions are satisfied in the case of the Nagumo equation (for all $a\in (0,1)$) and do not require any estimates 
on the unknown wave speed $c$.  

\medskip 
\noindent 
Our interest in the above reaction diffusion equation is motivated 
by the fact that \eqref{RDE} can be seen as a singular limit $\epsilon\downarrow 0$ of Fitz-Hugh Nagumo systems 
$$ 
\begin{aligned} 
\partial_t v (t,x) & = \nu v_{xx}(t,x) + b f(v(t,x)) - w(t,x) + I \\
\partial_t w (t,x) & = \varepsilon ( v (t,x) -\gamma w(t,x)) 
\qquad (t,x)\in\R_+\times\R 
\end{aligned} 
$$ 
when the adaptation variable $w$ is set constant to the value of the input current $I$ (see the monograph \cite{ET}). The Fitz-Hugh Nagumo system, a mathematical idealization of the Hodgkin Huxley model, 
admits, under appropriate assumptions on the coefficients, pulse 
solutions that serve as a mathematical model for the action potential travelling along the nerve axon. By adding noise to this system, e.g. channel noise, the resulting dynamical system exhibits many interesting features like propagation failure of the pulse solution, backpropagation, annihilation and spontaneous pulse solutions. Recent computational studies can be found in \cite{T2010, T2011}. 

\medskip 
\noindent 
We are therefore interested in a rigorous mathematical analysis of 
stochastic reaction-diffusion systems with bistable reaction terms. 
With a view towards the above mentioned features of the noisy system, 
we are in particular interested to establish a multiscale analysis of the whole dynamics which requires in a first step a robust 
stability result of the travelling pulse solution. As already 
mentioned for the scalar-valued case, the existing stability results (e.g. \cite{Ev, Jones} for systems) cannot be carried over to the stochastic case. In order to reduce the mathematical difficulty of the problem, we therefore consider the scalar-valued case in the present paper as a starting point.

\medskip 
\noindent 
Before we proceed let us first draw a couple of conclusions on the travelling wave resulting from our assumptions. 
\begin{lemma} 
\label{lem1_1}
Assume that {\bf (A1)} and {\bf (A2)} hold. Then:  
\begin{itemize}  
\item[(i)] $\hat{v}^2_x (x) \le \frac{2b}{\nu} \int^1_{\hat{v}(x)} f (v)\, dv$ for all $x$. In particular, 
$$ 
\lim_{x\to + \infty} e^{- \alpha \frac{c}{\nu}x} \hat{v}^2_x = 0 \qquad\mbox{ for } \alpha \ge 0\, .  
$$
\item[(ii)] $e^{- 2 \frac{c}{\nu}x} \hat{v}^2_x$ is increasing (resp. decreasing) for $x\le\hat{v}^{-1} (a)$ (resp. $x\ge \hat{v}^{-1} (a)$). 
In particular, 
$$ 
\lim_{x\to\pm\infty} e^{- \alpha\frac{c}{\nu} x } \hat{v}_x^2 = 0 \qquad\mbox{ for } \alpha \in\, [0, 2[ \, .  
$$ 
\end{itemize} 
\end{lemma} 

\medskip 
\noindent 
The proof of Lemma \ref{lem1_1} is given in Section \ref{Section0} below. The next Proposition summarizes the main conclusions implied by the additional assumption {\bf (A3)}.

\begin{proposition} 
\label{prop0} 
Assume that {\bf (A1)} - {\bf (A3)} hold. Then: 
\begin{itemize} 
\item[(i)] $\frac{f(\hat{v})}{\hat{v}_x}$ is strictly monotone increasing. In particular, 
$$ 
-\frac{d^2}{dx^2} \log{\hat{v}_x} = -\frac{d}{dx} \frac{\hat{v}_{xx}}{\hat{v}_x} =  \frac b\nu \frac{d}{dx}\frac{f(\hat{v})}{\hat{v}_x} > 0 \, , 
$$  
i.e., $\hat{v}_x$ is strictly log-concave (but not uniformly). 
\item[(ii)] 
$$ 
\begin{aligned} 
\gamma_- & := \inf \frac b\nu \frac{f(\hat{v})}{\hat{v}_x} = \frac{c}{2\nu} - \sqrt{\left(\frac{c}{2\nu}\right)^2 - \frac b\nu   f^\prime(0)} \\
\gamma_+ & := \sup \frac b\nu \frac{f(\hat{v})}{\hat{v}_x} = \frac{c}{2\nu} + \sqrt{\left(\frac{c}{2\nu}\right)^2 - \frac b\nu   f^\prime(1)}\, . 
\end{aligned} 
$$
\item[(iii)]  
$$
\begin{aligned} 
\int_{-\infty}^0  e^{-2\alpha\frac c\nu x}\left( \hat{v}_x^2 + \hat{v}_{xx}^2\right)\, dx < \infty \qquad\mbox{ for all } \alpha \frac c\nu  < \frac c\nu - \gamma_-  \\ 
\int_0^\infty     e^{-2\alpha\frac c\nu x}\left( \hat{v}_x^2 + \hat{v}_{xx}^2\right)\, dx < \infty \qquad\mbox{ for all } \alpha \frac c\nu  > \frac c\nu - \gamma_+  \\ 
\end{aligned} 
$$
In particular, 
$$ 
\int e^{-\frac c\nu x}\left( \hat{v}_x^2 + \hat{v}_{xx}^2\right)\, dx < \infty\, . 
$$ 
\end{itemize} 
\end{proposition} 

\medskip 
\noindent 
The proof of Proposition \ref{prop0} is given in Section \ref{Section0} below.

\bigskip 
\noindent 
The next theorem contains the essential functional inequality that is implied by {\bf (A3)}.  

\begin{theorem} 
\label{th0} 
Assume that {\bf (A1)} - {\bf (A3)} hold. Then there exists some $\kappa > 0$ such that 
\begin{equation} 
\label{FunctIneq}
- \frac{d^2}{dx^2} \log{\hat{v}_x} + \left( \frac d{dx} \log\hat{v}_x \right)^2 - \frac c\nu \frac d{dx}\log \hat{v}_x \ge \kappa\, . 
\end{equation} 
\end{theorem}

\medskip 
\noindent 
The proof of Theorem \ref{th0} is given in Section \ref{Section0} below.
We will assume from now on for all subsequent results that {\bf (A1)} - {\bf (A3)} hold. 

\begin{example} 
\label{ExampleNagumo}
In the particular case of the Nagumo equation, i.e., $f(v) = v(1-v)(v-a)$ for $a\in (0,1)$, the travelling wave is explicitely given as  
$\hat{v}(x) = (1+ e^{-kx})^{-1}$ (resp. its spatial translates) with $k = \sqrt{\frac{b}{2\nu}}$. The corresponding wave speed $c$ can be 
calculated as $c = \sqrt{2\nu b} \left( \frac 12 - a\right)$. The logarithmic derivative $\rho := \frac d{dx} \log\hat{v}_x  
= \frac{\hat{v}_{xx}}{\hat{v}_x}$ is given as $\rho = \frac c\nu - \frac b\nu \frac{f}{\hat{v}_x}  
= \sqrt{\frac{2b}{\nu}} \left( \frac 12 - \hat{v}\right)$. Thus  
$$ 
- \rho '  + \rho^2 - \frac c\nu \rho  = \frac b\nu \left( (\hat{v}-a)^2 + a(1-a) \right) \ge \frac b\nu a(1-a) > 0 \, . 
$$  
\end{example}

\bigskip 
\noindent 
With the functional inequality \eqref{FunctIneq} of Theorem \ref{th0} we can now state the mentioned result on the Lyapunov stability of the linearization of \eqref{RDE} 
along the travelling wave $\hat{v}$ in the deterministic case. To state our result precisely, let us introduce the Hilbert space $H = L^2 (\R)$ and the Sobolev 
space $V = H^{1,2} (\R)$, defined as the closure of $C_c^1 (\R)$ w.r.t. the norm 
$$ 
\|u\|_V^2 = \int_{\R} u^2  + u^2_x  \, dx 
$$ 
in $H$. Identifying $H$ with its dual $H^\prime$ we obtain dense and continuous embeddings $V\hookrightarrow H \equiv H^\prime 
\hookrightarrow H^\prime$. Note that w.r.t. this embedding the dualization between $V^\prime$ and $V$ reduces for $f\in H$ to 
the inner product in $H$, i.e., $_{V^\prime}\langle f, g\rangle_V = \langle f,g\rangle_H = \int fg \, dx$. 
The elementary estimate $u^2 (y) = 2\int_{-\infty}^y u_x (x) u(x) \, dx \le \int u_x^2 + u^2\, dx 
\le\|u\|_V^2$ for $u\in C_c^1 (\R )$ can be extended to the estimate $\|u\|_\infty \le \|u\|_V$ for all $u\in V$ that turns out 
to be crucial in the following.

\smallskip 
\noindent 
The unbounded linear operator $\nu u_{xx}$ induces a continuous mapping $ A: V\rightarrow V'$, because for $u\in C_c^1 (\R)$   
$$ 
\begin{aligned}
_{V^\prime}\langle A u, v \rangle_{V^\prime} 
& = \int \nu u_{xx} \, v\,  dx = - \nu\int u_x v_x \, dx 
 \le \nu \| u \|_V \|v\|_V\, . 
\end{aligned}
$$

\begin{theorem} 
\label{th1} 
Let $u\in V$. Then 
$$ 
_{V^\prime}\langle Au + bf^\prime (\hat{v}) u, u\rangle_V \le - \kappa_\ast \|u\|_V^2 + C_\ast \langle u, \hat{v}_x \rangle^2  
$$ 
where 
$$ 
\kappa_\ast := \frac{\kappa}{\kappa + \left( \frac c{2\nu} \right)^2} \frac \nu{q_1} 
$$ 
and 
$$ 
C_\ast = \left( \kappa_\ast q_2 + \frac{\nu}{\kappa} \left(\frac c{2\nu}\right)^2 \left( \kappa + \left(\frac c{2\nu}\right)^2\right)  
\frac{\int e^{-\frac c\nu x }\hat{v}_x^2\, dx}{\left(\int e^{-\frac c{2\nu}x }\hat{v}_x^2\, dx \right)^2}  \right)\, . 
$$ 
Here, $\kappa$ is the lower bound obtained in Theorem \ref{th0} and $q_1$ and $q_2$ are defined in Lemma \ref{lem2_1} below. 
\end{theorem} 

\noindent 
The proof of Theorem \ref{th1} is given in Section \ref{Proofth1} below.

\medskip 
\noindent 
The previous Theorem states that the flow generated by the semilinear diffusion equation is contracting in the direction that is orthogonal to 
$\hat{v}_x$ (and its spatial translates).  To properly quantify this contraction  we will need to model the equation \eqref{RDE} 
as an evolution equation in the appropriate function space.

\subsection{Realization of \eqref{RDE} as evolution equation}

\medskip 
\noindent 
In the next step we want to realize the reaction diffusion equation \eqref{RDE} as an evolution equation on a suitable function space. 
To this end we need to impose yet additional assumptions on the reaction term, but now concerning only its global behaviour at infinity 
and not affecting its behaviour on $[0,1]$ hence also not the travelling wave $\hat{v}$. We assume that the derivative $f^\prime$ of 
the reaction term is bounded from above 
\begin{equation} 
\tag{\bf B1} 
\eta_1 := \sup_{x\in\R} f^\prime (x) < \infty   \, , 
\end{equation} 
that there exists a finite positive constant $L$ such that 
\begin{equation} 
\tag{\bf B2} 
\left| f(x_1 ) - f(x_2)\right| \le L |x_1 - x_2| \left( 1 + x_1^2 + x_2^2 \right) \qquad \forall x_1 , x_2 \in \R\, , 
\end{equation} 
which is typically satisfied for polynomials of third degree with leading negative coefficient and that there exists $\eta_2$ such that  
\begin{equation}
\tag{\bf B3} 
\left| f(u+v) - f(v) - f^\prime (v) u \right|  \le \eta_2 (1 + |u|) |u|^2 \qquad \forall v\in [0,1]\, , u\in \R\, . 
\end{equation}

\smallskip 
\noindent 
Since we are interested in the asymptotic stability of the travelling wave also w.r.t. stochastic perturbations, it is now natural to consider the following 
decomposition $v(t,x) = u(t,x) + \hat{v} (x)$ of the solution $v$ of \eqref{RDE}, where $u$ now satisfies the following equation 
\begin{equation} 
\label{WREL} 
u_t (t,x) = \nu u_{xx} (t,x) + b\left( f(u(t,x) + \hat{v}(x)) - f(\hat{v}(x)\right) 
\end{equation} 
on $\R_+\times \R$ that can be analysed best in a variational framework.

\section{The deterministic case} 

\smallskip 
\noindent 
The nonlinear term 
\begin{equation} 
\label{Nonlinear}
G (t,u) : = f (u + \hat{v}(t)) - f (\hat{v}(t))
\end{equation} 
can be realized as a continuous mapping 
$$ 
G : [0, \infty )\times V \rightarrow V^\prime 
$$ 
being Lipschitz w.r.t. second variable $u$ on bounded subsets of $V$. Indeed, condition {\bf (B2)} on $f$ implies that  
$$ 
\begin{aligned} 
_{V^\prime}\langle G (t,u), w \rangle_{V^\prime} 
& = \int_{\R} G (t,u) w\, dx = \int_{\R} \left(f(u + \hat{v}(t)) - f(\hat{v}(t))\right) w\, dx \\ 
& \le  L \int_{\R} |u| (2 + u^2)|w|\, dx \le L \|u\|_H \left( 3 + 2\|u\|^2_V\right) \|w\|_H 
\end{aligned} 
$$ 
hence 
\begin{equation} 
\label{Est1}
\begin{aligned}  
\| G (t,u) \|_{V^\prime} \leq L \|u\|_H \left( 3 + 2\|u\|^2_V \right) 
\end{aligned} 
\end{equation} 
and similarly 
$$
\begin{aligned}
_{V^\prime} \langle G (t,u_1) - G(t,u_2), w \rangle_{V^\prime}  
& = \int_{\R} (f (u_1 + \hat{v} (t)) - f (u_2 + \hat{v}(t) ) w \, dx \\
& \le L \|u_1 - u_2\|_H \left( 4 + 2\| u_1\|^2_V + 2\| u_2 \|^2_V \right) \|w\|_H 
\end{aligned}
$$ 
which implies 
\begin{equation} 
\label{Est2}  
\| G (t,u_1) - G (t,u_2) \|_{V^\prime} \le 2L \left( 2 + \|u_1\|_V^2 +  \|u_2\|^2_V\right) \| u_1 - u_2\|_H  \, . 
\end{equation} 
The sum $Au + bG (t,u)$ of both operators now satisfies the global monotonicity condition  
\begin{equation} 
\label{Est3} 
\begin{aligned} 
\langle A u_1 & + b G (t,u_1) - A u_2 - b G (t,u_2), u_1 - u_2 \rangle \\
& = \int A(u_1 - u_2)(u_1 - u_2)\, dx + b\int (G(t,u_1) - G(t,u_2))(u_1 - u_2)\, dx \\  
& = - \nu \int (u_1 - u_2)_x^2\, dx 
   + b\int (G(t,u_1) - G(t,u_2))(u_1 - u_2)\, dx \\ 
& \le b\eta_1 \| u_1 - u_2 \|^2_H  
\end{aligned} 
\end{equation} 
using {\bf (B1)} and similarly the coercivity condition 
\begin{equation} 
\label{Est4} 
\langle Au + bG(t,u), u \rangle \le - \nu \|u\|^2_V + 
(\nu +  b \eta_1 ) \|u\|^2_H 
\end{equation} 
since $f(s)s = (f(s) - f(0))(s-0) \le\eta_1 s^2$ for all $s\in\R$  using {\bf (B1)}.

\smallskip
\noindent  
Theorem 1.1 in \cite{LR} now implies for all initial conditions $u_0\in H$ and all finite times $T$ existence and uniqueness of a variational solution 
$u\in L^\infty ([0,T];H)\cap L^2 ([0,T]; V)$ satisfying the integral equation  
\begin{equation} 
\label{VSDet} 
u (t) = u_0 + \int^t_0 \left( A u (s) + b (f (u(s) + \hat{v}(s)) - f (\hat{v}(s)))\right)\, ds 
\end{equation} 
and we may extend the solution to the whole time axes $\R_+$. 

\smallskip 
\noindent 
The integral on the right hand side of \eqref{VSDet} is well-defined as a Bochner integral in $L^2 ([0,T]; V^\prime )$ using \eqref{Est1} which implies in 
particular that the mapping $t \mapsto u (t)$, $\R_+ \rightarrow V^\prime$, is differentiable with differential 
\begin{equation} 
\label{RDDet} 
\frac{du}{dt} = Au (t) + b \left( f (u (t) + \hat{v}(t)) - f ( \hat{v}(t))\right) \quad \in V^\prime \, ,  
\end{equation} 
hence continuous. 

\smallskip 
\noindent 
We are now ready to state precisely our notion of stability we are going to prove in the following. 

\begin{definition} 
\label{Defi1}
The travelling wave solution $\hat{v}$ is called locally asymptotically stable w.r.t. the $H$-norm if there exists $\delta > 0$ such that for initial condition 
$v_0$ with $v_0 - \hat{v} \in H$ and $\|v_0 - \hat{v}\|_H\le \delta$ the unique variational solution $u(t,x) = v(t,x) - \hat{v}(x)$ of 
\eqref{WREL} satisfies
$$ 
\lim_{t\to\infty} \| v_0 - \hat{v}(\cdot + x_0)\|_H = 0 
$$ 
for some (phase) $x_0\in\R$.  
\end{definition} 

\medskip 
\noindent 
In order to apply Theorem \ref{th1} we need to control the tangential component $\langle v(t) - \hat{v}(\cdot + x_0),\hat{v}_x (\cdot + x_0)\rangle^2$ 
of the given solution $v(t) = u(t) + \hat{v}(t)$ w.r.t. the appropriate phase-shift $x_0$, i.e., the phase-shift $x_0$ that minimizes the  $L^2$-distance between the solution $v(t)$ and 
the orbit consisting of all phase-shifted travelling waves $\hat{v} (\cdot + x_0 )$. This can be achieved asymptotically by introducing 
dynamically by by introducing the following ordinary differential equation 
\begin{equation} 
\label{ODE} 
\begin{aligned} 
\dot{C}(t) 
& = c + m\langle v(t) - \hat{v} \left( \cdot + C(t)\right), \hat{v}_x \left( \cdot + C(t)\right) \rangle  \, , \\ 
C(0) & = 0 
\end{aligned}
\end{equation} 
for $m \ge 0$. To simplify notations, let 
$$ 
\tilde{v} (t) := \hat{v} (\cdot + C(t)) 
$$ 
so that we can rewrite equation \eqref{ODE} as 
\begin{equation} 
\label{ODE1} 
\begin{aligned} 
\dot{C}(t) 
& = c + m \langle v(t) - \tilde{v}(t) , \tilde{v}_x (t)\rangle\, , \\ 
C(0) & = 0\, .   
\end{aligned}
\end{equation} 
The next Proposition first shows that \eqref{ODE} is well-posed. 

\begin{proposition}
\label{prop3_4} 
Let $v = u + \hat{v} (t)$ be a solution of \eqref{RDDet} with $u\in L^\infty ([0, T ], H) \cap L^2 ([0, T ]; V )$. Then
$$ 
B(t, C) = \langle v (t)- \hat{v}(\cdot + C), \hat{v}_x (\cdot + C)\rangle_H 
$$ 
is continuous in $(t, C) \in [0,T]\times\R$ and Lipschitz continuous w.r.t. $C$ with Lipschitz constant independent of $t$.
\end{proposition} 

\begin{proof}
First note that 
$$
\begin{aligned}
B(t, C_1) - B (t, C_2) 
& = \langle \hat{v}_x ( \cdot + C_1) - \hat{v}_x ( \cdot + C_2), u (t)\rangle_H \\
& \qquad - \langle \hat{v}_x (\cdot + C_1), \hat{v} ( \cdot + C_1) - \hat{v} ( \cdot )\rangle_H \\                    
& \qquad + \langle \hat{v}_x (\cdot + C_2), \hat{v} ( \cdot + C_2) - \hat{v} ( \cdot )\rangle_H 
\end{aligned}
$$ 
Using 
$$ 
\begin{aligned}  
\hat{v}_x (x + C_1) - \hat{v}_x (x + C_2) 
& = \int_{C_1}^{C_2} \hat{v}_{xx} (x + y)\, dy \\ 
& \le \int_{C_1}^{C_2} |\hat{v}_{xx}| (x + y) \, dy 
\end{aligned} 
$$ 
we conclude that the first term on the right hand side can be estimated from above by
$$ 
\begin{aligned}
\|\hat{v}_x ( \cdot + C_1) & - \hat{v}_x ( \cdot + C_2)\|_H  \|u(t)\|_H \\
& \le \left( |C_1 - C_2| \int_{\R} \int_{C_1}^{C_2} \hat{v}_{xx}^2 (x + y) \, dy\, dx \right)^{\frac 12} \|u(t)\| _H  \\
& = |C_1 - C_2 | \|\hat{v}_{xx}\|_H^2 \|u(t)\|_H 
\end{aligned} 
$$ 
which implies that this term is Lipschitz continuous with Lipschitz constant independent of $t \in [0,T]$.

\medskip 
\noindent 
The second and the third term can be rewritten as follows: 
$$
\begin{aligned}
& \Big| \langle \hat{v}_x ( \cdot + C_1), \hat{v} (\cdot + C_1) - \hat{v} ( \cdot )\rangle_H \\
& \qquad\qquad  - \langle \hat{v}_x ( \cdot + C_2), \hat{v} ( \cdot + C_2) - \hat{v} (\cdot )\rangle_H \Big| \\
& \qquad = \Big| \langle \hat{v}_x , \hat{v}(\cdot - C_2) - \hat{v} ( \cdot - C_1 )\rangle_H \Big| \\
& \qquad \le  \|\hat{v}_x\|_H  \left( \int_{\R} \left( \int_{C_1}^{C_2} \hat{v}_x ( \cdot + y )\, dy \right)^2 \, dx \right)^{\frac 12} \\
& \qquad \le  \|\hat{v}_x\|_H  |C_1 - C_2| \|\hat{v}_x\|_H 
\end{aligned}
$$ 
so that also these two terms are Lipschitz continuous with Lipschitz constant independent of $t$.  
\end{proof}

\medskip 
\noindent 
In the following let
\begin{equation} 
\label{tildeU} 
\tilde{u} (t) := u(t) + \hat{v}(t) - \tilde{v}(t) = v(t) - \tilde{v} (t)\, . 
\end{equation}

\begin{proposition}
\label{prop1_3} 
Let $u = v - \hat{v}(t) \in L^\infty \left( [0,T]; H \right) \cap L^2 \left( [0,T]; V\right)$ be a solution of \eqref{RDDet} and  
$\tilde{u}$ be given by \eqref{tildeU}. Then $\tilde{u}\in L^\infty \left( [ 0,T]; H  \right) \cap  L^2 \left( [0,T ]; V\right)$ 
again and $\tilde{u}$ satisfies the evolution equation 
\begin{equation} 
\label{RDDet1}
\begin{aligned}
\frac{d\tilde{u}}{dt} (t)  
& = \nu \Delta \tilde{u}(t) + b \tilde{G} \left( t, \tilde{u}(t)\right) - (\dot{C} (t) -c) \tilde{v}_x (t)  \\
& = \nu \Delta \tilde{u} (t) + b f' \left( \tilde{v} (t) \right) \tilde{u} (t) + b \tilde{R} \left( t, \tilde{u} (t)\right)  - (\dot{C} (t) -c)\tilde{v}_x (t)
\end{aligned} 
\end{equation} 
with 
$$ 
\begin{aligned} 
\tilde{G}(t,u) & = f \left( u + \tilde{v} (t)\right) - f \left( \tilde{v}(t)\right) \, , \\
\tilde{R} (t,u) &  = \tilde{G} (t,u) - f' \left( \tilde{v} (t)\right) u \, . 
\end{aligned}
$$ 
\end{proposition}

\noindent
The proof of the Proposition is an immediate consequence of \eqref{RDDet} and \eqref{ODE} (resp. \eqref{ODE1}). {\bf (B3)} implies 
for the remainder 
$\tilde{R}$ the following estimate 
\begin{equation} 
\label{remainder}  
\begin{aligned} 
\langle \tilde{R} (t,u),u\rangle 
& \le \eta_2 \int (1 + |u|)|u|^3\, dx 
\le \eta_2 \left( \|u\|_\infty + \|u\|_\infty^2 \right) \|u\|_H^2 \\ 
& \le \eta_2 \left( \|u\|_H + \|u\|_H^2 \right) \|u\|_V^2 \, . 
\end{aligned} 
\end{equation}

\medskip 
\noindent 
We are now ready to state our main result in the deterministic case: 

\begin{theorem}
\label{th2} 
Recall the definition of $\kappa_\ast$ and $C_\ast$ in Theorem \ref{th1}. Let $m\ge C_\ast$. If the initial condition 
$v_0 = u_0 + \hat{v}$ is close to $\hat{v}$ in the sense that 
$$  
\|u_0\|_H < \left(\delta \frac{\kappa_\ast}{2b\eta_2}\right)\wedge 1
$$ 
for some $\delta < 1$ and $v(t) = u(t) + \hat{v} (t)$, where $u(t)$ is the unique solution of \eqref{RDDet}, then 
$$ 
\|v(t) - \hat{v} \left( \cdot + C(t)\right)\|_H \le e^{-(1-\delta)\kappa_\ast t} \|v_0 -\hat{v}\|_H\, . 
$$ 
\end{theorem}

\begin{proof} 
Let $\tilde{u} (t) := v(t) - \tilde{v} (t)$ be as in \eqref{tildeU}. Then Proposition \ref{prop1_3} and equation \eqref{remainder} imply that 
\begin{equation} 
\label{EstTh2_1} 
\begin{aligned} 
\frac 12 \frac{d}{dt} \|\tilde{u}(t)\|^2_H 
& = \langle \nu \Delta\tilde{u}(t) + bf^\prime (\tilde{v}(t)) \tilde{u} (t), \tilde{u}(t)\rangle 
     + b\langle \tilde{R} (t, \tilde{u} (t)) , \tilde{u} (t) \rangle \\ 
& \quad - m\langle \tilde{v}_x(t) , \tilde{u} (t)\rangle^2 \\ 
& \le \langle \nu \Delta\tilde{u}(t) + bf^\prime (\tilde{v}(t)) \tilde{u} (t), \tilde{u}(t)\rangle \\
& \quad + b\eta_2 \left( \|\tilde{u} (t)\|_H + \|\tilde{u}(t)\|_H^2 \right) \|\tilde{u}(t)\|_V^2  
   - m \langle \tilde{v}_x (t) , \tilde{u} (t)\rangle^2\, . 
\end{aligned} 
\end{equation} 
Using translation invariance of $\nu \Delta$ and $\int u_x^2\, dx$, Theorem \ref{th1} yields the estimate  
\begin{equation} 
\label{EstTh2_2} 
\begin{aligned} 
\langle \nu \Delta\tilde{u}(t) & + bf^\prime (\tilde{v}^{TW}(t)) \tilde{u} (t), \tilde{u}(t)\rangle \\ 
& \le - \kappa_\ast \|\tilde{u} (t)\|_V^2 
+ C_\ast \langle \tilde{u}(t), \tilde{v}_x\rangle^2 \, .   
\end{aligned} 
\end{equation} 
Inserting \eqref{EstTh2_2} into \eqref{EstTh2_1} yields that 
$$ 
\begin{aligned} 
\frac 12 \frac d{dt} \|\tilde{u} (t) \|_H^2 
& \le - \kappa_\ast \|\tilde{u} (t) \|^2_V  + b\eta_2 \left( \|\tilde{u} (t)\|_H + \|\tilde{u}(t)\|_H^2 \right) \|\tilde{u} (t) \|_V^2 \, . 
\end{aligned} 
$$ 

\medskip 
\noindent 
In the next step we define the stopping time 
$$ 
T := \inf \left\{ t\ge 0 \mid \|\tilde{u}(t)\|_H \ge \left( \delta \frac{\kappa_\ast}{2b\eta_2}\right) \wedge 1\right\}  
$$ 
with the usual convention $\inf \emptyset = \infty$. Continuity of $t\mapsto \|\tilde{u}(t)\|_H$ implies that 
$T > 0$ since $\|u_0\|_H < \left( \delta \frac{\kappa_\ast}{2b\eta_2}\right)\wedge 1$. For $t < T$ note that 
$$ 
\frac 12 \frac d{dt} \|\tilde{u} (t) \|_H^2 \le - (1- \delta )\kappa_\ast \|\tilde{u} (t) \|^2_V  
\le - (1- \delta )\kappa_\ast \|\tilde{u} (t) \|^2_H 
$$ 
which implies that 
$$ 
\|\tilde{u}(t)\|^2_H \le e^{-2(1-\delta )\kappa_\ast t} \|u_0\|^2_H  
$$ 
for $t < T$. Suppose now that $T < \infty$. Then continuity of $t\mapsto \|\tilde{u}(t)\|_H$ implies on the one hand 
that $\|\tilde{u}(T)\|_H = \left( \delta \frac{\kappa_\ast}{2b\eta_2} \right)\wedge 1$ and on the other hand, using the last inequality,   
$$ 
\|\tilde{u}(T)\|_H = \lim_{t\uparrow T} \|\tilde{u} (t)\|_H \le e^{-(1-\delta )\kappa_\ast T} \|u_0\|_H < \left( \delta \frac{\kappa_\ast}{2b\eta_2}\right)\wedge 1 
$$ 
which is a contradiction. Consequently, $T = \infty$ and thus 
$$
\|\tilde{u} (t)\|_H \le e^{-(1-\delta )\kappa_\ast t} \|u_0\|_H \qquad\forall t\ge 0 
$$ 
which implies the assertion. 
\end{proof}


\section{The reaction-diffusion equation with noise} 
\label{sec2}

\noindent 
In this section we will generalize the stability result for the reaction-diffusion equation \eqref{RDE} to the stochastic case. To this end 
we cosider the following equation 
\begin{equation} 
\label{StochRDE} 
\begin{aligned} 
dv (t) & = \left[ \nu\partial^2_{xx} v (t) +  bf(v(t))\right]\, dt  + \Sigma_0 (v(t))\, dW (t) 
\end{aligned} 
\end{equation} 
where $W = (W (t))_{t\ge 0}$ is a cylindrical Wiener process with values in some separable real Hilbert space $U$ defined on some underlying 
filtered probability space $(\Omega , \cF , (\cF (t))_{t\ge 0}, P)$ and 
$$ 
\Sigma_0 : \hat{v} + H \mapsto L_2 (U, H) 
$$ 
is a measurable map with values in the linear space of all Hilbert-Schmidt operators from $U$ to $H$ such that there exists some constant $L_{\Sigma_0}$ with 
\begin{equation} 
\label{dispersion} 
\| \Sigma_0 ( \hat{v} + u_1 ) - \Sigma_0 (\hat{v} + u_2)\|_{L_2 (U,H)} \le L_{\Sigma_0} \|u_1 - u_2\|_H  \quad\forall u_1 \, , u_2 \in H\, . 
\end{equation} 
For the theory of cylindrical Wiener processes see \cite{PR}. To simplify presentation of the results we also assume the following translation 
invariance 
\begin{equation} 
\label{TransInv} 
\|\Sigma_0 (\hat{v}(\cdot - C))\|_{L_2 (U, H)} 
= \|\Sigma_0 (\hat{v} + \left(\hat{v}(\cdot - C)- \hat{v}\right))\|_{L_2 (U, H)} \, \forall C\in\R\, . 
\end{equation} 

\medskip 
\noindent 
A typical example covered by the assumptions is 
$$ 
dv (t) = \left[ \nu\partial^2_{xx} v (t) +  bf(v(t))\right]\, dt  + \sigma(v(t))\, dW^Q (t) 
$$ 
where $\sigma : \R \to\R$ is Lipschitz, $\sigma (0) = \sigma (1) = 0$, 
$W^Q$ is a $Q$-Wiener process with covariance operator $Q$ for which 
its square-root $\sqrt{Q}$ admits a kernel $k_{\sqrt{Q}} (x,y)\in L^2 (\R^2)$ satisfying 
$$ 
\sup_{x\in\R} \int k^2_{\sqrt(Q)}(x,y)\, dy < \infty 
$$  
(see \cite{St}).

\medskip 
\noindent 
Similar to the deterministic case we can give the equation a rigorous formulation as a stochastic evolution equation with values in the Hilbert space $H = L^2 (\R)$ by decomposing $v(t) = u(t) + \hat{v} (t)$ w.r.t. the travelling wave to obtain the following stochastic evolution equation 
\begin{equation} 
\label{StochRelRDE} 
\begin{aligned} 
du(t) & = \left[ \nu\Delta u (t) + bG(t, u(t))\right]\, dt + \Sigma (t , u (t))\, dW(t) 
\end{aligned} 
\end{equation} 
where the nonlinear term $G$ is as in \eqref{Nonlinear} and   
\begin{equation} 
\label{RealizationDispersion}
\Sigma (t,u)h := \Sigma_0 \left( \hat{v}(t) + u\right) h \, , \quad u\in H\, , h\in U\, ,    
\end{equation} 
is a continuous mapping 
$$ 
\Sigma (\cdot , \cdot ) : [ 0, \infty ) \times H\to L_2 (U,H) \, . 
$$ 
The assumptions \eqref{dispersion} and \eqref{TransInv} on the dispersion operator imply   
\begin{equation}
\label{Dispersion1} 
\| \Sigma (t, u_1 ) - \Sigma (t, u_2)\|_{L_2 (U,H)} \le L_{\Sigma_0}  \|u_1 - u_2\|_H 
\end{equation} 
and 
\begin{equation} 
\label{Dispersion2} 
\| \Sigma (t, u)\|_{L_2 (U,H)} \le \| \Sigma_0 (\hat{v})\|_{L_2 (U,H)}  + L_{\Sigma_0} \|u\|_H \, . 
\end{equation} 

\medskip 
\noindent 
We now consider the equation \eqref{StochRDE} w.r.t. the same triple $V\hookrightarrow H \equiv H^\prime\hookrightarrow V^\prime$ 
as in the deterministic case. Due to the properties \eqref{Est1}, \eqref{Est2}, \eqref{Est3} and \eqref{Est4}, 
we can deduce from Theorem 1.1. in \cite{LR} for all finite $T$ and all (deterministic) initial conditions $u_0\in H$ the 
existence and uniqueness of a solution $(u(t))_{t\in [0, T]}$ of \eqref{StochRDE} satisfying the moment estimate 
$$ 
E\left( \sup_{t\in[0,T]} \|u(t)\|^2_H + \int_0^T \|u(t)\|_V^2\, dt \right) < \infty \, . 
$$ 
In particular, for any $m\in\R$, we can apply Proposition \ref{prop3_4} to a typical trajectory $u(\cdot )(\omega )$ to obtain a unique solution $C(\cdot )(\omega )$ 
of the ordinary differential equation \eqref{ODE1}. It is also clear that the resulting stochastic process $(C(t))_{t\ge 0}$ is $(\cF_t )_{t\ge 0}$-adapted, 
since $(u(t))_{t\ge 0}$ is. 

\medskip 
\noindent 
In the next step let us consider the stochastic process 
$$ 
\tilde{u} (t) = u(t) + \hat{v}(t) - \hat{v}(\cdot + C(t))  = v(t) -\tilde{v}(t)  
$$ 
which is $(\cF_t)_{t\ge 0}$ adapted too and satisfies the stochastic evolution equation 
$$ 
d\tilde{u} (t) = \left[ \nu\Delta \tilde{u} (t) + b \tilde{G} (t, \tilde{u}(t))  
- (\dot{C} (t) -c)\tilde{v}_x (t)\right] \, dt + \tilde{\Sigma} (t, \tilde{u} (t))\, dW(t) \, ,  
$$
where 
$$ 
\tilde{G} (t,u) = f(u + \tilde{v} (t) ) -  f(\tilde{v}(t)) \, , 
\tilde{\Sigma} (t,u) = \Sigma (t, u + \tilde{v} (t)) \, , 
$$ 
and the moment estimates 
$$ 
E\left( \sup_{t\in [0,T]} \|\tilde{u} (t)\|_H^2 + \int_0^T \|\tilde{u} (t)\|^2_V\, dt \right) < \infty\, . 
$$ 
Theorem 4.2.5 in \cite{PR} now implies that the real-valued stochastic process $\|\tilde{u}\|^2_H (t)$ is a continuous local semimartingale so that 
we have in particular the following time-dependent Ito-formula 
\begin{equation} 
\label{ItoFormula} 
\begin{aligned} 
\varphi (t,\|\tilde{u}(t)\|^2_H ) 
& = \int_0^t \varphi_t (s, \|\tilde{u} (s)\|^2_H) + 2\varphi_x (s,\|\tilde{u} (s)\|^2_H) 
\langle \nu\Delta \tilde{u} (s)   \\
& \quad + b\tilde{G}(s, \tilde{u} (s)) - \dot{C} (s) \tilde{w} (s), \tilde{u} (s) \rangle  \\
& \quad + \varphi_x (s,\|\tilde{u} (s)\|^2_H)\|\tilde{\Sigma} (s,\tilde{u} (s))\|^2_{L_2(H)} \\ 
& \quad + \varphi_{xx}  (s,\|\tilde{u} (s)\|_H^2) 2\|\tilde{\Sigma}^\ast (s,\tilde{u} (s))\tilde{u} (s)\|_H^2\, ds \\
& \quad  + \int_0^t \varphi_x (s, \|\tilde{u} (s)\|^2_H )\, d\tilde{M}_s 
\end{aligned} 
\end{equation} 
for any $\varphi\in C^{1,2}([0,T]\times \R_+)$. Here, $\tilde{\Sigma}^\ast (s,u)$ denotes the adjoint operator of $\tilde{\Sigma} (s,u)$.

\begin{theorem} 
\label{th3} 
Recall the definition of $\kappa_\ast$ and $C_\ast$ in Theorem \ref{th1} and assume that $L^2_{\Sigma_0} \le \frac{\kappa_\ast} 4$. 
Let $v_0 = u_0 + \hat{v}$ and $v(t) = u(t) + \hat{v} (t)$, where $u(t)$ is the unique solution of the stochastic evolution equation 
\eqref{StochRDE} and $\tilde{u}(t) = u(t) + \hat{v} (t) - \tilde{v}(t)$. Then 
$$ 
P\left( T < \infty \right) \le \frac 1{c^2_\ast} \left( \|\tilde{u} (0)\|^2_H + \frac 4{\kappa_\ast} \|\Sigma_0 (\hat{v})\|^2_{L_2 (U,H)} \right) 
$$ 
where $T$ denotes the first exit time 
\begin{equation} 
\label{ExitTime} 
T := \inf\{t\ge 0\mid \|\tilde{u} (t) \|_H > c_\ast \} \, , \qquad c_\ast = \left( \frac{\kappa_\ast}{4b\eta_2}\right)\wedge 1\, , 
\end{equation} 
with the usual convention $\inf\emptyset = \infty$. 
\end{theorem}

\begin{proof} 
Similar to the proof of Theorem \ref{th1} we have the following inequality 
$$
\begin{aligned} 
& \langle \nu\Delta \tilde{u} (t) + b\tilde{G}(t, \tilde{u} (t)) 
- (\dot{C} (t) - c) \tilde{v}_x  (t), \tilde{u} (t) \rangle \\
& 
\qquad \le -\kappa_\ast \|\tilde{u} (t)\|^2_V + b\eta_2 \left(\|\tilde{u}(t)\|_H + \|\tilde{u}(t)\|_H^2 \right) 
\|\tilde{u}(t)\|^2_V\, . 
\end{aligned} 
$$
In particular,  
$$ 
\langle \nu\Delta \tilde{u} (t) + b\tilde{G}(t, \tilde{u} (t)) 
- (\dot{C} (t)-c) \tilde{v}_x (t), \tilde{u} (t) \rangle 
\le -\frac{\kappa_\ast}2 \|\tilde{u} (t)\|^2_V 
$$
for $t\le T$, where $T$ is as in \eqref{ExitTime}. \eqref{Dispersion2} 
and \eqref{TransInv} imply 
$$  
\|\tilde{\Sigma} (\tilde{u} (t))\|_{L_2 (H)}^2 \le 2 \left( L_{\Sigma_0}^2 \|\tilde{u}(t)\|^2_H  + \|\Sigma_0 (\hat{v})\|^2_{L_2 (U,H)} \right) 
$$ 
and therefore 
$$
\begin{aligned} 
& 2\langle \nu\Delta \tilde{u} (t) + b\tilde{G}(t, \tilde{u} (t)) 
- \dot{C} (t) \tilde{v}_x (t), \tilde{u} (t) \rangle + \|\tilde{\Sigma} (t,\tilde{u} (t))\|^2_{L_2 (H)} \\ 
\qquad & \le -\frac{\kappa_\ast}2 \|\tilde{u} (t)\|^2_V + 2 
\|\Sigma_0 (\hat{v})\|^2_{L_2 (U,H)}\, . 
\end{aligned} 
$$ 
Applying Ito's formula \eqref{ItoFormula} to $e^{\frac {\kappa_\ast} 2 t}x$, then yields for $t < T$ that 
$$ 
\begin{aligned} 
e^{\frac{\kappa_\ast} 2 t}\|\tilde{u} (t)\|^2_H 
& \le \|\tilde{u}(0)\|^2_H + \frac 4{\kappa_\ast} \left( e^{\frac{\kappa_\ast} 2 t} - 1\right)\|\Sigma_0 (\hat{v})\|^2_{L_2 (U,H)} \\
& \qquad + \int_0^t e^{\frac{\kappa_\ast} 2 s}\, d\tilde{M}_s\, . 
\end{aligned} 
$$ 
Taking expectations we obtain  
$$ 
E\left( \|\tilde{u} (t\wedge T )\|^2_H \right) \le  \|\tilde{u}(0)\|^2_H + \frac 4{\kappa_\ast}\|\Sigma_0 (\hat{v})\|^2_{L_2 (U,H)} 
$$ 
and thus in the limit $t\uparrow\infty$ 
$$ 
\begin{aligned} 
c_\ast^2 P\left( T < \infty \right) 
& =  E\left( \|\tilde{u} (T)1_{T < \infty} \|^2_H \right) \le \lim_{t\uparrow\infty} E\left( \|\tilde{u}(t\wedge T) \|^2_H \right)  \\
& \le  \|\tilde{u}(0)\|^2_H + \frac 4{\kappa_\ast} \|\Sigma_0 (\hat{v})\|^2_{L_2 (U,H)} 
\end{aligned} 
$$ 
which implies the assertion. 
\end{proof}


\section
{Proof of Lemma \ref{lem1_1}, Proposition \ref{prop0}  and Theorem \ref{th0}} 
\label{Section0}

\subsection{Proof of Lemma \ref{lem1_1} and Proposition \ref{prop0}}

\begin{proof} {\bf (of Lemma \ref{lem1_1})}   
For the proof of (i) note that $\hat{v}_x \ge 0$ and $\int^{\infty}_{-\infty} \hat{v}_x dx = \lim_{x\to \infty } \hat{v} (x) - \hat{v} (-x) = 1$. 
In particular, $\hat{v}_x \in L^1 (\R)$ which implies that $\lim_{n\to\infty} \hat{v}_x (x_n) = 0$ for some sequence $x_n\uparrow\infty$. It follows for all $x$ that
$$ 
\begin{aligned} 
\hat{v}^2_x (x) 
& = \hat{v}_x^2 (x_n) - 2 \int^{x_n}_{x} \hat{v}_{xx} \hat{v}_x \, dx \\
& = \hat{v}^2_x (x_n) - 2 \frac{c}{\nu} \int^{x_n}_{x}\hat{v}^2_x \, dx + 2 \frac{b}{\nu} \int^{x_n}_x f (\hat{v}) \hat{v}_x \, dx \\ 
& \le \hat{v}_x^2 (x_n) + 2 \frac{b}{\nu} \int^{\hat{v} (x_n)}_{\hat{v} (x)} f(v) \, dv  \qquad \forall n\, .  
\end{aligned} 
$$ 
Consequently, 
$$ 
\hat{v} _x^2 (x) 
\le \lim_{n\to\infty} \hat{v}^2_x (x_n) + 2 \frac{b}{\nu} \int^{\hat{v} (x_n)}_{\hat{v} (x)} f (v)\, dv 
= \frac{2b}{\nu} \int^1_{\hat{v} (x)} f(v)\, dv \, . 
$$
In particular, 
$$ 
\lim_{x\to\infty} \hat{v}_x^2 (x) \le \limsup_{x\to\infty} \frac{2b}{\nu} \int^1_{\hat{v} (x)} f (v)\, dv = 0 
$$ 
and thus also $\lim_{x\to\infty} e^{ - \alpha\frac{c}{\nu}x} \hat{v}^2_x (x) = 0$ for all $\alpha\ge 0$.

\medskip 
\noindent 
For the proof of (ii) note that for all $\alpha \in \R$ 
\begin{equation}
\label{eq_lem1_1_1} 
\begin{aligned} 
\frac{d}{dx} ( e ^{- \alpha \frac{c}{\nu} x } \hat{v}^2_x ) 
& = \left( - \alpha \frac{c}{\nu}\hat{v}_x  + 2 \hat{v}_{xx}\right) e^{- \alpha \frac{c}{\nu} x} \hat{v}_x \\ 
& = (2 - \alpha) \frac{c}{\nu} e ^{- \alpha\frac{c}{\nu} x} \hat{v}^2_x - \frac{b}{\nu} e^{- \alpha \frac{c}{\nu} x } f (\hat{v}) \hat{v}_x\, . 
\end{aligned} 
\end{equation}
Taking $\alpha = 2 $ we conclude in particular that $\frac{d}{dx} \left( e^{- \alpha \frac{c}{\nu} x } \hat{v}^2_x \right) \ge 0$ 
(resp. $\le 0$) for $x \le v^{-1} (a)$ (resp. $x\ge v^{-1} (a)$), since $v_x \ge 0$ and $f(\hat{v} (x)) \le 0 $ (resp. $\ge 0$) 
for $x\le v^{-1} (a)$ (resp. $x\ge v^{-1} (a)$). Consequently, for $c\ge 0$, 
$$ 
\lim_{x\to-\infty} e^{- 2\frac{c}{\nu} x } \hat{v}^2_x (x) = \inf_{x \le v^{-1} (a)} e^{- 2\frac{c}{\nu} x } \hat{v}^2_x (x) =: \gamma < \infty 
$$ 
and thus for $\alpha < 2$  
$$ 
\lim_{x\to -\infty} e^{- \alpha\frac{c}{\nu} x} \hat{v}^2_x (x) \le \limsup_{x\to -\infty} e^{(2 - \alpha ) \frac{c}{\nu} x } \gamma = 0\, . 
$$ 
Similarly in the case $c\le 0$ 
$$ 
\lim_{x\to\infty} e^{- 2\frac{c}{\nu} x } \hat{v}^2_x (x) = \inf_{x \ge v^{-1} (a)} e^{- 2\frac{c}{\nu} x } \hat{v}^2_x (x) =: \gamma < \infty 
$$ 
and thus for $\alpha < 2$  
$$ 
\lim_{x\to\infty} e^{- \alpha\frac{c}{\nu} x} \hat{v}^2_x (x) \le \limsup_{x\to\infty} e^{(2 - \alpha ) \frac{c}{\nu} x } \gamma = 0\, . 
$$ 
Combining with (i) we obtain the assertion. 
\end{proof}

\bigskip 
\noindent 
Let us now turn to the proof of Proposition \ref{prop0}. Let $x_\ast = \hat{v}^{-1} (v_\ast )$ and $w(x) := e^{-\frac c{2\nu} x} \hat{v}_x (x)$. Then 
$$ 
w_{xx} = \left( \left( \frac c{2\nu}\right)^2 - \frac b\nu f'(\hat{v})\right) w\, , 
$$ 
since differentiating $c\hat{v}_x = \nu \hat{v}_{xx} + bf(\hat{v})$ implies $c\hat{v}_{xx} = \nu \hat{v}_{xxx} + bf' (\hat{v})\hat{v}_x$. 

\medskip 
\noindent 
{\bf Proof of Proposition \ref{prop0} (i)} Note that 
$$ 
\frac d{dx} \left( w_x^2 + \left( \frac b\nu f' (\hat{v}) - \left( \frac c{2\nu}\right)^2 \right) w^2\right) = \frac b\nu f'' \left( \hat{v}\right) \hat{v}_x w^2 
$$ 
is strictly increasing (resp. decreasing ) for $x < x_\ast $ (resp. $x > x_\ast $). According to Lemma \ref{lem1_1}  
$$
\lim_{|x| \to\infty}  \left( w_x^2 + \left( \frac b\nu f' (\hat{v}) - \left( \frac c{2\nu}\right)^2 \right) w^2\right) = 0 
$$ 
so that 
$$ 
w_x^2 + \left( \frac b\nu f' (\hat{v}) - \left( \frac c{2\nu}\right)^2 \right) w^2 \ge 0 \qquad \forall x\, . 
$$ 
Using $w_x = \left( \frac{c}{2 \nu} - \frac{b}{\nu} \frac{f(\hat{v})}{\hat{v}} \right) w$, we conclude that  
$$ 
\left( \frac{c}{2 \nu} - \frac{b}{\nu} \frac{f (\hat{v})}{\hat{v} _x } \right)^2 + \frac{b}{\nu} f^{\prime} (\hat{v}) - \left( \frac{c}{2 \nu} \right)^2 >  0 \, . 
$$ 
or equivalently 
\begin{equation} 
\label{th0:eq1} 
\frac{b}{\nu} f^{\prime} (\hat{v}) - \frac{b}{\nu} \frac{f (\hat{v})}{\hat{v} _x}\left( \frac c\nu - \frac{b}{\nu} \frac{ f(\hat{v})}{\hat{v}_x}\right) > 0 \, . 
\end{equation} 
In particular, 
$$ 
\begin{aligned} 
\frac b\nu \frac{d}{dx} \frac{f(\hat{v})}{\hat{v}_x} 
& = \frac b\nu f^{\prime} (\hat{v}) - \frac b\nu \frac{f (\hat{v})}{\hat{v}_x} \frac{\hat{v}_{xx}}{\hat{v}_x} > 0   
\end{aligned} 
$$
so that $\frac{f(\hat{v})}{\hat{v}_x}$ is strictly increasing which implies that $\hat{v}_x$ is log-concave, because 
$$ 
- \frac{d^2}{dx^2} \log\hat{v}_x = -\frac{d}{dx} \frac{\hat{v}_{xx}}{\hat{v}_x} = -\frac{d}{dx} \left( \frac c\nu - \frac b\nu \frac{f(\hat{v})}{\hat{v_x}} \right) > 0\, . 
$$ 

\medskip 
\noindent 
For the proof of part (ii) of Proposition \ref{prop0} we will first need the following 

\begin{lemma} 
\label{lem0_01} 
Let $K_+ := \frac {1-\hat{v}(x_0)}{\hat{v}_x (x_0)}$ and $K_- := \frac {\hat{v}(x_0)}{\hat{v}_x(x_0)}$. Then 
\begin{itemize} 
\item[(i)] $\frac {1-\hat{v}(x)}{\hat{v}_x (x)} \le K_+$ for $x\ge x_0$, 
\item[(ii)] $\frac {\hat{v}(x)}{\hat{v}_x (x)} \le K_-$ for $x\le x_0$. 
\end{itemize}
\end{lemma}  

\begin{proof} (i) Consider the function $h := \frac{ 1-\hat{v}}{\hat{v}_x}$. Clearly, 
$\dot{h} = - 1 - \frac{\hat{v}_{xx}}{\hat{v}_x} h$ is negative, hence $h$ decreasing, in a neighborhood of $x_0$. Since $\hat{v}_x$ is 
log-concave it follows that $-\frac{\hat{v}_{xx}}{\hat{v}_x}$ is increasing on $[x_0 , \infty)$. We may assume in the following 
that there exists some $x_+ > x_0$ with  
$$ 
-\frac{\hat{v}_{xx}}{\hat{v}_x} (x_+) = \frac{\hat{v}_x}{1-\hat{v}} (x_+) \, . 
$$ 
In fact, if this is not the case, then $\dot{h} \le 0$ for all $x\ge x_0$, hence $h$ decreasing on $[x_0 , \infty )$ which  
already implies the assertion. 

\smallskip 
\noindent 
So let us assume that $h$ is decreasing on $[x_0 , x_+]$ only. In particular, 
$\frac {1-\hat{v}(x)}{\hat{v}_x (x_+)} \le K_+$. For $x\ge x_+$ it follows that 
$-\frac{\hat{v}_{xx}}{\hat{v}_x} (x) \ge - \frac{\hat{v}_{xx}}{\hat{v}_x} (x_+) = \frac{\hat{v}_x}{1-\hat{v}} (x_+)$, hence 
$\frac{d}{dx} \left( e^{-\frac{\hat{v}_x}{1-\hat{v}} (x_+) x} \hat{v}_x \right) \le 0$, and consequently, 
$$ 
\begin{aligned} 
1-\hat{v}(x) 
& = \int_x^\infty \hat{v}_x(s)\, ds 
= \int_x^\infty  e^{\frac{\hat{v}_x}{1-\hat{v}} (x_+) s}  \left( e^{-\frac{\hat{v}_x}{1-\hat{v}}(x_+) (s)} \hat{v}_x(s) \right) \, ds \\ 
& \le \int_x^\infty  e^{\frac{\hat{v}_x}{1-\hat{v}} (x_+)s} \, ds \left( e^{-\frac{\hat{v}_x}{1-\hat{v}}(x_+) x} \hat{v}_x(x) \right) \\ 
& = \frac{\hat{v}}{1-\hat{v}_x}(x_+) \hat{v}_x (x) \le K_+ \hat{v}_x (x)\, . 
\end{aligned} 
$$ 

\smallskip 
\noindent 
(ii) is shown similar. 
\end{proof} 

\medskip 
\noindent 
{\bf Proof of Proposition \ref{prop0} (ii)} Since $f(0) = 0$ it follows that $\lim_{v\to 0} \frac{|f(v)|}{v} < \infty$ and thus 
$$ 
\limsup_{x\to -\infty} \frac{|f(\hat{v})|}{\hat{v}_x} (x) = \limsup_{x\to -\infty} \frac{|f(\hat{v})|}{\hat{v}} \frac{\hat{v}}{\hat{v}_x} (x) < \infty
$$ 
due to the previous Lemma \ref{lem0_01}. Similarly, $f(1) = 0$ implies that $\lim_{v\to 1} \frac{f(v)}{1-v} < \infty$ and thus 
$$ 
\limsup_{x\to\infty} \frac{f(\hat{v})}{\hat{v}_x} (x) = \limsup_{x\to\infty} \frac{f(\hat{v})}{1-\hat{v}} \frac{1-\hat{v}}{\hat{v}_x} (x) < \infty\, . 
$$
To compute $\gamma_-$ note that $\frac b\nu \frac{f(\hat{v})}{\hat{v}_x}$ is increasing in $x$, hence $\gamma_{-} = \lim_{x\to -\infty} \frac b\nu\frac{f(\hat{v})}{\hat{v}_x} (x) 
= \inf_{x\in\R} \frac b\nu \frac{f(\hat{v})}{\hat{v}_x}(x)$ exists, must be strictly negative and is finite. Applying l'Hospital's rule we obtain that 
$$ 
\gamma_- = \lim_{x\to -\infty} \frac b\nu \frac{f(\hat{v})}{\hat{v}_x}(x) = \lim_{x\to -\infty} \frac b\nu f^\prime (\hat{v})(x) \frac{\hat{v}_x}{\hat{v}_{xx}} (x) 
= \frac b\nu f^{\prime} (0) \frac 1{\frac c\nu - \gamma_{-}} 
$$ 
or equivalently, $\gamma_{-} \left(\frac c\nu - \gamma_{-}\right) =  \frac b\nu f^\prime (0)$. Since $\gamma_{-} < 0$ we obtain the assertion. $\gamma_+$ can be computed similarly. 

\medskip 
\noindent 
{\bf Proof of Proposition \ref{prop0} (iii)} The previous part implies for the logarithmic derivative of $\hat{v}_x$ that 
$$ 
\lim_{x\to -\infty}\frac{\hat{v}_{xx}}{\hat{v}_x} = \frac c\nu - \gamma_- > \frac c\nu 
$$
and 
$$ 
\lim_{x\to\infty}\frac{\hat{v}_{xx}}{\hat{v}_x} = \frac c\nu - \gamma_+ < \frac c\nu 
$$ 
so that for every $\alpha$ satisfying  $\alpha \frac c\nu < \frac c\nu - \gamma_-$ (resp.  $\alpha \frac c\nu > \frac c\nu - \gamma_+$) it follows that 
$e^{-\alpha\frac c\nu x}\hat{v}_x$ is increasing for small $x$ (resp. decreasing for large $x$). Hence $\int_{-\infty}^0 e^{-\alpha\frac c\nu x}\hat{v}_x^2\, dx < \infty$ 
(resp. $\int_0^\infty e^{-\alpha\frac c\nu x}\hat{v}_x^2\, dx < \infty$) in both cases.  

\medskip 
\noindent 
We can also now estimate 
$$ 
\int_{-\infty}^0 e^{-\alpha\frac c\nu x} \hat{v}_{xx}^2\, dx \le \sup_{x\in\R} \frac{|\hat{v}_{xx}|}{\hat{v}_x} \int_{-\infty}^0 e^{-\alpha\frac c\nu x} \hat{v}_x^2\, dx < \infty 
$$ 
for $\alpha\frac c\nu < \frac c\nu -\gamma_-$ and 
$$ 
\int_0^\infty e^{-\alpha\frac c\nu x} \hat{v}_{xx}^2\, dx \le \sup_{x\in\R} \frac{|\hat{v}_{xx}|}{\hat{v}_x} \int_0^\infty e^{-\alpha\frac c\nu x} \hat{v}_x^2\, dx < \infty 
$$ 
for $\alpha\frac c\nu > \frac c\nu -\gamma_+$, since 
$$ 
\sup_{x\in\R} \frac{|\hat{v}_{xx}|}{\hat{v}_x} \le \frac {|c|}\nu + \sup_{x\in\R} \frac {|f(\hat{v})|}{\hat{v}_x} < \infty 
$$ 
again due to the previous part (ii).

\subsection{Proof of Theorem \ref{th0}}

\bigskip 
\noindent 
Inequality \eqref{FunctIneq} is equivalent to   
\begin{equation} 
\label{Dissipativity1}
\frac b\nu f^\prime (\hat{v}) + 2\frac b\nu \frac{ f(\hat{v})}{\hat{v}_x} 
\left( \frac b\nu \frac{ f(\hat{v})}{\hat{v}_x} - \frac c\nu\right) \ge \kappa\, . 
\end{equation} 
Since 
$$ 
\frac b\nu f^\prime (\hat{v}) + 2\frac b\nu \frac{ f(\hat{v})}{\hat{v}_x} 
\left( \frac b\nu \frac{ f(\hat{v})}{\hat{v}_x} - \frac c\nu\right) 
> \frac b\nu \frac{ f(\hat{v})}{\hat{v}_x} 
\left( \frac b\nu \frac{ f(\hat{v})}{\hat{v}_x} - \frac c\nu\right) 
$$ 
and $\lim_{x\to\pm\infty} \frac b\nu \frac{ f(\hat{v})}{\hat{v}_x} 
\left( \frac b\nu \frac{ f(\hat{v})}{\hat{v}_x} - \frac c\nu\right) > 0$, it remains to prove 
that 
\begin{equation} 
\label{th0:eq2} 
g_2  := \frac 12 f^\prime (\hat{v})\hat{v}_x^2   -  f(\hat{v}) \hat{v}_{xx}  > 0 
\end{equation} 
for $x$ with $0\le \frac b\nu \frac{f(\hat{v})}{\hat{v}_x}(x)\le \frac c\nu$ in order to be able 
to find $\kappa > 0$ satisfying \eqref{Dissipativity1}. In the particular case $c = 0$ this is obvious.

\medskip 
\noindent 
We therefore assume from now on that $c > 0$. Since $\frac{f(\hat{v})}{\hat{v}_x}$ is strictly increasing, 
it follows that for all 
$\alpha\in ]\inf \frac bc \frac{f(\hat{v})}{\hat{v}_x} ,  \sup \frac bc \frac{f(\hat{v})}{\hat{v}_x} [$ there exists a unique $x_\alpha\in\R$ with 
$$ 
\frac{b}{\nu} \frac{f (\hat{v})}{\hat{v} _x} (x_\alpha ) = \alpha \frac c\nu\, . 
$$ 
In particular, $\hat{v}(x_0) = a$ and $\hat{v}(x_1) $ is the unique root of $\hat{v}_{xx}$, that is, $x_1$ is the location of the maximum of $\hat{v}_x$ and 
$x_0 \le x_1$ and both, $f(\hat{v})$, $f^\prime(\hat{v}) \ge 0$ on $[x_0, x_1]$.

\medskip 
\noindent
We will subdivide the proof of \eqref{th0:eq2} into the three cases 
$x\in[x_0  , x_{0.5}\wedge x_\ast ]$, $x\in [x_{0.5}\vee x_\ast , x_1]$ and $x\in [x_{0.5}\wedge x_\ast , x_{0.5}\vee x_\ast ]$.  

\begin{lemma}
\label{th0:lem1} 
$g_2 (x) > 0$ for $x\in [x_0 , x_{0.5}\wedge x_\ast ]$. 
\end{lemma} 

\begin{proof} 
We may suppose that $x_\ast \ge x_0$, because otherwise, the interval is empty. Let 
$$ 
\bar{x} := \inf \{ x\ge x_0 \mid g_2 (x) = 0\} \, . 
$$ 
We will show that $\bar{x} > x_{0.5}\wedge x_\ast$. Since $g_2 (x_0) = \frac 12 f^\prime (a)\hat{v}_x^2 (x_0) > 0$ we certainly have that $\bar{x} > x_0$. Suppose now that $\bar{x}\le x_{0.5}\wedge x_\ast$. Then for all $m\in\N$ 
$$ 
\begin{aligned} 
\frac d{dx} f(\hat{v})\hat{v}_{xx} \hat{v}_x^m 
& = f^\prime (\hat{v})  \hat{v}_{xx} \hat{v}_x^{1+m} + f (\hat{v}) \left( \hat{v}_{xxx} 
+ m\frac{\hat{v}_{xx}^2}{\hat{v}_x} \right) \hat{v}_x^{m} \\ 
& = \frac c{\nu}   f(\hat{v}) \hat{v}_{xx}  \hat{v}_x^{m}  +  f^\prime (\hat{v}) \left( \frac c\nu - 2\frac b\nu \frac{f(\hat{v})}{\hat{v}_x}   \right) \hat{v}_x^{m+2}  \\
& \qquad    + mf (\hat{v})\hat{v}^2_{xx} \hat{v}_x^{m-1} 
\end{aligned} 
$$ 
which implies 
\begin{equation}
\label{th0:lem1-1} 
\begin{aligned} 
f(\hat{v})\hat{v}_{xx}\hat{v}_x^{m} (\bar{x})  
& = \int_{x_0}^{\bar{x}}  e^{\frac c\nu (\bar{x}-s)} f^\prime (\hat{v}) \left( \frac c\nu - 2\frac b\nu \frac{f(\hat{v})}{\hat{v}_x} \right) \hat{v}_x^{m+2}  \, ds \\ 
& \qquad  + m\int_{x_0}^{\bar{x}}  e^{\frac c\nu (\bar{x}-s)} f(\hat{v}) \hat{v}^2_{xx}  \hat{v}_x^{m-1} \, ds \\ 
& =:  I + II\, , \quad\mbox{say.}
\end{aligned} 
\end{equation} 
Now $f^{(2)}(\hat{v})\ge 0$, hence $f^\prime (\hat{v})$ increasing, and $\frac c\nu - 2\frac b\nu \frac{f(\hat{v})}{\hat{v}_x} \ge 0$ due to    
$x\le x_{0.5}\wedge x_\ast$, implies that 
$$
\begin{aligned} 
I & \le f^\prime(\hat{v}) (\bar{x})  \int_{x_0}^{\bar{x}}  
e^{\frac c\nu (\bar{x}-s)} \left( \frac c\nu - 2\frac b\nu \frac{f(\hat{v})}{\hat{v}_x} \right) \hat{v}_x^{m+2} \, ds \\ 
& = \frac 12  f^\prime (\hat{v}) (\bar{x}) \left( \hat{v}_x^{m+2} (\bar{x}) - e^{\frac c\nu (\bar{x}-x_0)}\hat{v}_x^{m+2} (x_0) 
- m \int_{x_0}^{\bar{x}} e^{\frac c\nu (\bar{x}-s)} \hat{v}_{xx}\hat{v}_x^{m+1} \, ds \right) \\
& \quad + \frac 12  f^\prime (\hat{v}) (\bar{x}) \int_{x_0}^{\bar{x}} e^{\frac c\nu (\bar{x}-s)} \left( \frac c{2\nu} - \frac b\nu \frac{f(\hat{v})}{\hat{v}_x} \right) \hat{v}_x^{m+2} \, ds  
\end{aligned} 
$$ 
thereby using $e^{\frac c\nu (\bar{x}-s)}\left( \frac c\nu \hat{v}_x 
- 2\frac b\nu f(\hat{v})\right) \hat{v}_x  = \frac d{ds} e^{\frac c\nu (\bar{x}-s)}\hat{v}_x^2$. Inserting the last estimate into \eqref{th0:lem1-1} and using $g_2 (s) \ge 0$ for $s\le \bar{x}$, hence 
$$ 
II \le \frac m2 \int_{x_0}^{\bar{x}} e^{\frac c\nu (\bar{x}-s)} 
f^\prime(\hat{v}) \hat{v}_{xx} \hat{v}_x^{m+1}\, ds\, ,  
$$ 
we arrive at  
$$ 
\begin{aligned} 
f(\hat{v})\hat{v}_{xx}\hat{v}_x^{m} (\bar{x})  
& < \frac 12  f^\prime (\hat{v})\hat{v}_x^{m+2} (\bar{x}) \\
& \qquad - \frac m2   \int_{x_0}^{\bar{x}} e^{\frac c\nu (\bar{x}-s)} \left( f^\prime(\hat{v}) (\bar{x}) - f^\prime (\hat{v}) (s)\right)  
\hat{v}_{xx}\hat{v}_x^{m+1} \, ds \\
&  \qquad + \frac 12 f^\prime (\hat{v}) (\bar{x})  \int_{x_0}^{\bar{x}}  e^{\frac c\nu (\bar{x}-s)} \left( \frac c{2\nu} - \frac b\nu \frac{f(\hat{v})}{\hat{v}_x}\right)\hat{v}_x^{m+2} \, ds  \, . 
\end{aligned} 
$$ 
We can now choose $m$ sufficiently large such that 
$$ 
\begin{aligned} 
f^\prime (\hat{v}) & (\bar{x})  \int_{x_0}^{\bar{x}}  e^{\frac c\nu (\bar{x}-s)} \left( \frac c{2\nu} - \frac b\nu \frac{f(\hat{v})}{\hat{v}_x} \right)\hat{v}_x^{m+2} \, ds \\  
& < m \int_{x_0}^{\bar{x}} e^{\frac c\nu (\bar{x}-s)} \left( f^\prime(\hat{v}) (\bar{x}) - f^\prime (\hat{v}) (s)\right)  
\hat{v}_{xx}\hat{v}_x^{m+1} \, ds 
\end{aligned} 
$$ 
since $f^\prime (\hat{v}) (\bar{x}) - f^\prime (\hat{v}) (s) > 0$ for $s < \bar{x}$. It follows that $f(\hat{v})\hat{v}_{xx}\hat{v}_x^{m} (\bar{x}) <  \frac 12 f^\prime(\hat{v})\hat{v}_x^{m+2}(\bar{x})$, which is a contradiction to the definition of $\bar{x}$. It follows that $\bar{x} > x_{0.5}\wedge x_\ast$ and thus $g_2 (x) > 0$ on $[x_0 , x_{0.5}\wedge x_\ast]$. 
\end{proof}

\medskip 
\noindent 
We now turn to the second subinterval $[x_{0.5}\vee x_\ast , x_1]$ where $f^\prime (\hat{v})$ decreases. 

\begin{lemma}
\label{th0:lem2} 
$g_2 (x) > 0$ for $x\in [x_{0.5} \vee x_\ast , x_1]$. 
\end{lemma} 

\begin{proof} 
We may assume that $x_\ast \le x_1$. Otherwise the interval $[x_0 \vee x_\ast  , x_1]$ is empty. Let 
$$ 
\bar{x} := \sup \{ x\in [x_{0.5}\vee x_\ast , x_1] \mid g_2 (x) = 0\} \, . 
$$ 
In this case we will show that $\bar{x} < x_{0.5}\vee x_\ast$. Since $g_2 (x_1) = \frac 12 f^\prime (v)\hat{v}_x^2 (x_1) > 0$ we certainly have that $\bar{x} < x_1$. 
Suppose now that $\bar{x}\ge x_{0.5}\vee x_\ast$. Then for all $m\in\N$ we have that 
$$ 
\begin{aligned} 
\frac d{dx} f(\hat{v})\hat{v}_{xx} \hat{v}_x^{-m} 
& = f^\prime (\hat{v})  \hat{v}_{xx} \hat{v}_x^{1-m} + f (\hat{v}) \left( \hat{v}_{xxx} 
- m\frac{\hat{v}_{xx}^2}{\hat{v}_x} \right) \hat{v}_x^{-m} \\ 
& = \frac c{\nu}   f(\hat{v}) \hat{v}_{xx}  \hat{v}_x^{-m}  -  f^\prime (\hat{v}) \left( 2\frac b\nu f(\hat{v}) - \frac c\nu \hat{v}_x \right) \hat{v}_x^{1-m}  \\
& \qquad    - mf (\hat{v})\hat{v}^2_{xx} \hat{v}_x^{-(m+1)} 
\end{aligned} 
$$ 
which implies 
\begin{equation}
\label{th0:lem2-1} 
\begin{aligned} 
f(\hat{v})\hat{v}_{xx}\hat{v}_x^{-m} (\bar{x})  
& = \int_{\bar{x}}^{x_1}  e^{\frac c\nu (\bar{x}-s)} f^\prime (\hat{v}) \left( 2\frac b\nu f(\hat{v}) - \frac c\nu \hat{v}_x \right) \hat{v}_x^{1-m} \, ds \\
& \qquad  + m\int_{\bar{x}}^{x_1}  e^{\frac c\nu (\bar{x}-s)} f(\hat{v}) \hat{v}^2_{xx}  \hat{v}_x^{-(m+1)} \, ds \\ 
& =:  I + II\, , \quad\mbox{say.}
\end{aligned} 
\end{equation} 
Now $f^{(2)}(\hat{v})\le 0$, hence $f^\prime (\hat{v})$ decreasing, and $2\frac b\nu f(\hat{v}) - \frac c\nu \hat{v}_x \ge 0$ due to $x\ge x_{0.5}\vee x_\ast$, implies that 
$$
\begin{aligned} 
I & \le f^\prime(\hat{v}) (\bar{x})  \int_{\bar{x}}^{x_1}  e^{\frac c\nu (\bar{x}-s)} \left( 2\frac b\nu \frac{f(\hat{v})}{\hat{v}_x} - \frac c\nu \right) \hat{v}_x^{2-m} \, ds \\ 
& = \frac 12  f^\prime (\hat{v}) (\bar{x}) \left( \hat{v}_x^{2-m} (\bar{x}) - e^{\frac c\nu (\bar{x}-x_1)}\hat{v}_x^{2-m} (x_1) 
- m \int_{\bar{x}}^{x_1}  e^{\frac c\nu (\bar{x}-s)} \hat{v}_{xx}\hat{v}_x^{1-m} \, ds \right) \\
& \quad + \frac 12  f^\prime (\hat{v}) (\bar{x}) \int_{\bar{x}}^{x_1} e^{\frac c\nu (\bar{x}-s)} \left( \frac b\nu \frac{f(\hat{v})}{\hat{v}_x} - \frac c{2\nu} \right) \hat{v}_x^{2-m} \, ds  \\
\end{aligned} 
$$ 
thereby using $e^{\frac c\nu (\bar{x}-s)}\left( 2\frac b\nu f(\hat{v}) - \frac c\nu \hat{v}_x \right) \hat{v}_x  = - \frac d{ds} e^{\frac c\nu (\bar{x}-s)}\hat{v}_x^2$. Inserting the last estimate into \eqref{th0:lem2-1} and using $g_2 (s) \ge 0$ for $s\ge \bar{x}$, hence 
$$ 
II \le \frac m2 \int_{\bar{x}}^{x_1} e^{\frac c\nu (\bar{x}-s)} 
f^\prime(\hat{v}) \hat{v}_{xx} \hat{v}_x^{1-m}\, ds\, ,  
$$ 
we arrive at  
$$ 
\begin{aligned} 
f(\hat{v})\hat{v}_{xx}\hat{v}_x^{-m} (\bar{x})  
& < \frac 12  f^\prime (\hat{v})\hat{v}_x^{2-m} (\bar{x}) \\
& \qquad - \frac m2   \int_{\bar{x}}^{x_1}  e^{\frac c\nu (\bar{x}-s)} \left( f^\prime(\hat{v}) (\bar{x}) - f^\prime (\hat{v}) (s)\right)  
\hat{v}_{xx}\hat{v}_x^{1-m} \, ds \\
&  \qquad + \frac 12 f^\prime (\hat{v}) (\bar{x})  \int_{\bar{x}}^{x_1}  e^{\frac c\nu (\bar{x}-s)} \left( \frac b\nu \frac{f(\hat{v})}{\hat{v}_x}   - \frac c{2\nu}\right)\hat{v}_x^{2-m} \, ds  \, . 
\end{aligned} 
$$ 
We can now choose $m$ sufficiently large such that 
$$ 
\begin{aligned} 
f^\prime (\hat{v}) & (\bar{x})  \int_{\bar{x}}^{x_1}  e^{\frac c\nu (\bar{x}-s)} \left( \frac b\nu \frac{f(\hat{v})}{\hat{v}_x} - \frac c{2\nu}\right)\hat{v}_x^{2-m} \, ds \\  
& < m \int_{\bar{x}}^{x_1}  e^{\frac c\nu (\bar{x}-s)} \left( f^\prime(\hat{v}) (\bar{x}) - f^\prime (\hat{v}) (s)\right)  
\hat{v}_{xx}\hat{v}_x^{1-m} \, ds 
\end{aligned} 
$$ 
since $f^\prime (\hat{v}) (\bar{x}) - f^\prime (\hat{v}) (s) > 0$ for $s > \bar{x}$. It follows that $f(\hat{v})\hat{v}_{xx}\hat{v}_x^{-m} (\bar{x}) <  \frac 12 f^\prime(\hat{v})\hat{v}_x^{2-m}(\bar{x})$, which is a contradiction to the definition of $\bar{x}$. It follows that $\bar{x} < x_{0.5}\vee x_\ast$ and thus $g_2 (x) > 0$ on $[x_{0.5}\vee x_\ast , x_1 \vee x_\ast]$. 
\end{proof}

\noindent 
Finally we consider the third subinterval $[x_{0.5}\wedge x_\ast , x_{0.5}\vee x_\ast]$. 

\begin{lemma}
\label{th0:lem3} 
$g_2 (x) > 0$ for $x\in [x_{0.5}\wedge x_\ast , x_{0.5}\vee x_\ast]$. 
\end{lemma} 

\begin{proof} 
We consider the two cases $x_\ast \le x_{0.5}$ and $x_{0.5} > x_\ast$ separately. 

\medskip 
\noindent 
{\bf Case 1:} $x_\ast \le x_{0.5}$, hence 
$[x_{0.5}\wedge x_\ast , x_{0.5}\vee x_\ast] = [x_\ast , x_{0.5}]$.

\medskip 
\noindent 
In this case $f^\prime (\hat{v})$ is decreasing and 
$\frac{f(\hat{v})}{\hat{v}_x}\frac{\hat{v}_{xx}}{\hat{v}_x}$ increases, since 
$$ 
\frac d{dx} \frac{f(\hat{v})}{\hat{v}_x}\frac{\hat{v}_{xx}}{\hat{v}_x} = \left( \frac c\nu - 2\frac b\nu  \frac{f(\hat{v})}{\hat{v}_x} \right) \frac  d{dx} \frac{f(\hat{v})}{\hat{v}_x} \ge 0\, . 
$$ 
Hence 
$$ 
\begin{aligned} 
g_2 (x) & = \hat{v}_x^2 (x) \left( \frac 12 f^\prime (\hat{v}) - \frac{f(\hat{v})}{\hat{v}_x}\frac{\hat{v}_{xx}}{\hat{v}_x}\right) (x) \\ 
& \ge \hat{v}_x^2 (x) \left( \frac 12 f^\prime (\hat{v}) - \frac{f(\hat{v})}{\hat{v}_x}\frac{\hat{v}_{xx}}{\hat{v}_x}\right) (x_{0.5}) \\ 
& \ge \frac{\hat{v}_x^2 (x)}{\hat{v}_x^2 (x_{0.5})} g_2 (x_{0.5}) > 0   
\end{aligned} 
$$
according to Lemma \ref{th0:lem2}.

\medskip 
\noindent 
{\bf Case 2:} $x_{0.5} < x_\ast$, hence 
$[x_{0.5}\wedge x_\ast , x_{0.5}\vee x_\ast] = [x_{0.5} , x_\ast]$.

\medskip 
\noindent 
In this case $f^\prime (\hat{v}$ is increasing and 
$\frac{f(\hat{v})}{\hat{v}_x}\frac{\hat{v}_{xx}}{\hat{v}_x}$ decreases, since 
$$ 
\frac d{dx} \frac{f(\hat{v})}{\hat{v}_x}\frac{\hat{v}_{xx}}{\hat{v}_x} = \left( \frac c\nu - 2\frac b\nu  \frac{f(\hat{v})}{\hat{v}_x} \right) \frac  d{dx} \frac{f(\hat{v})}{\hat{v}_x} \le 0\, . 
$$ 
Hence 
$$ 
\begin{aligned} 
g_2 (x) & = \hat{v}_x^2 (x) \left( \frac 12 f^\prime (\hat{v}) - \frac{f(\hat{v})}{\hat{v}_x}\frac{\hat{v}_{xx}}{\hat{v}_x}\right) (x) \\ 
& \ge \hat{v}_x^2 (x) \left( \frac 12 f^\prime (\hat{v}) - \frac{f(\hat{v})}{\hat{v}_x}\frac{\hat{v}_{xx}}{\hat{v}_x}\right) (x_{0.5}) \\ 
& \ge \frac{\hat{v}_x^2 (x)}{\hat{v}_x^2 (x_{0.5})} g_2 (x_{0.5} ) > 0   
\end{aligned} 
$$
according to Lemma \ref{th0:lem1}.  
\end{proof}


\section
{Proof of Theorem \ref{th1}}
\label{Proofth1}

\noindent 
Recall that the travelling wave satisfies the equation $c\hat{v}_x = \nu\hat{v}_{xx} + bf (\hat{v})$, 
hence $c\hat{v}_{xx} = \nu\hat{v}_{xx} + b f ^{\prime} (\hat{v}) \hat{v}_x$. Given a function $u\in C^1_c (\R)$ 
and writing $u = h \hat{v}_x$ it follows that  
$$ 
\nu\Delta u + b f^\prime (\hat{v}) u = \nu h_{xx}  \hat{v}_x + 2\nu\hat{v}_{xx} h_x + c\hat{v}_{xx}h 
$$ 
which implies 
\begin{equation}
\label{eq2_1}
\begin{aligned} 
& -\langle \nu\Delta u + b f^\prime (\hat{v}) u, u \rangle 
  = -\int (\nu h_{xx} + 2\nu\frac{\hat{v}_{xx}}{\hat{v}_x} hx ) \ h \, \hat{v}_x^2 dx - c\int h\hat{v}_{xx} h \hat{v}_x dx \\
& \qquad =  \nu\int h^2_x \hat{v}^2_x \, dx + c\int h_x h \hat{v}^2_x dx \\  
& \qquad =  \nu  \int \left( h e^{\frac c{2\nu}x}\right)^2_x e^{-\frac c\nu x} \hat{v}^2_x\, dx - \nu\left(\frac{c}{2\nu} \right)^2  \int h^2 \hat{v}^2_x dx \\ 
& \qquad =: \cE (h) \, . 
\end{aligned}
\end{equation} 
In the following, consider the two functions $h_0 (x) = 1$ and $h_1 (x) = e^{-\frac c{2\nu} x}$. Notice that $\cE (h_0) = 0$ and 
$\cE (h_1) = \nu\left(\frac{c}{2\nu} \right)^2  \int e^{-\frac c\nu x} \hat{v}^2_x dx > 0$. Consequently, the Schr\"odinger operator  $\nu \Delta u + bf^\prime (\hat{v})u$ 
is not negative definite on the subspace $\cN := \mbox{ span} \{ \hat{v}_x , e^{-\frac c{2\nu} x}\hat{v}_x \}$. $\hat{v}_x$ can be interpreted as the vector pointing in the tangential 
direction of the orbit of the travelling wave solutions, since $\frac{d}{dt} \hat{v} (\cdot + ct) = c \hat{v}_x (\cdot + ct)$ and the second function 
$h_1 (x) = e^{-\frac c{2\nu} x}$ measures the infinitesimal variation of the linearization of $\nu \Delta u + b \left( f(u+ \hat{v}) - f(\hat{v})\right)$ w.r.t. time. 
Notice that in the case $c=0$ of a stationary wave both functions coincide, since the linearization is independent of the time. 

\medskip 
\noindent 
Using the representation \eqref{eq2_1} we will now first consider the gradient form $\int h_x^2 w^2 \, dx$, where $w = e^{-\frac c{2\nu} x}\hat{v}_x$. The logarithmic derivative 
$$ 
\theta (x) := \frac{w_x (x)}{w(x)} = \frac c{2\nu} - \frac b\nu \frac{f(\hat{v})}{\hat{v}_x} (x)  = \frac{\hat{v}_{xx}}{\hat{v}_x} - \frac c{2\nu}
$$ 
of $w$ satisfies the inequality 
$$ 
\begin{aligned} 
- \theta '  + \theta^2 
& = - \frac d{dx} \frac{\hat{v}_{xx}}{\hat{v}_x} + \left( \frac{\hat{v}_{xx}}{\hat{v}_x} - \frac c{2\nu}\right)^2 \\ 
& = - \frac d{dx} \frac{\hat{v}_{xx}}{\hat{v}_x} + \left( \frac{\hat{v}_{xx}}{\hat{v}_x}\right)^2 -\frac c\nu \frac{\hat{v}_{xx}}{\hat{v}_x} + \left( \frac c{2\nu}\right)^2   
\ge \kappa +  \left( \frac c{2\nu}\right)^2  
\end{aligned} 
$$ 
for some $\kappa > 0$ according to Theorem \ref{th0}. Proposition \ref{PropHardy} below now implies the weighted Hardy type inequality 
\begin{equation} 
\label{WeightedHardy} 
\int h^2 w^2 \, dx \le \frac 1{\kappa +  \left( \frac c{2\nu}\right)^2} \int h_x^2\, w^2\, dx 
\end{equation} 
for any $h\in C_b^1 (\R)$ with $h(x_{0.5}) = 0$, where $x_{0.5}$ is the unique root of $\frac b\nu \frac{f(\hat{v}}{\hat{v}_x} (x) = \frac c{2\nu}$ (recall that 
$\frac{f(\hat{v})}{\hat{v}_x}$ is strictly monotone increasing). Clearly, the last inequality implies the Poincare inequality  
\begin{equation}
\label{Poincare}
\int h^2 w^2\, dx \le \frac 1{\kappa +  \left( \frac c{2\nu}\right)^2}  \int h_x^2 w^2\, dx + Z^{-1} \left( \int hw^2\, dx \right)^2 
\end{equation}
for the normalizing constant $Z = \int e^{-\frac c\nu x} \hat{v}_x^2\, dx$ and for any $h\in C_b^1 (\R)$. Unfortunately, this is not yet enough, since for $u = he^{\frac c{2\nu} x}\hat{v}_x$ 
we cannot control the tangential direction $\int h w^2\, dx = \int u e^{-\frac c{2\nu} x}\hat{v}_x\, dx$ but only the tangential direction  $\int h e^{\frac c{2\nu} x} w^2\, dx = \int u \hat{v}_x\, dx$. 
This is done in the following 

\begin{proposition} 
\label{prop2_1a} 
For $h\in C_b^1 (\R)$ the following inequality holds: 
\begin{equation}
\label{Step1-eq1} 
\int h^2\, w^2\, dx \le \frac 1{\kappa + \left( \frac c{2\nu}\right)^2} \int h_x^2\, w^2\, dx + C_\ast \left( \int he^{\frac c{2\nu}x}w^2\, dx \right)^2 \, . 
\end{equation}
with  
$$ 
C_{\ref{prop2_1a}} = \frac {\kappa + \left( \frac c{2\nu}\right)^2}\kappa  \frac{\int e^{-\frac c\nu x}\hat{v}_x^2\, dx}
{\left( \int e^{-\frac c{2\nu}x} \hat{v}_x^2\, dx\right)^2} \, . 
$$ 
\end{proposition} 

\medskip 
\noindent 
The proof of Proposition \ref{prop2_1a} requires the following lemma.

\begin{lemma} 
\label{lem1a} 
There exists a function $g\in C^1 (\R)\cap L^2 (\R, w^2\, dx)$, $g\ge 0$, satisfying the equation 
\begin{equation} 
\label{lem1a_eq1}
\left( \kappa + \left( \frac c{2\nu}\right)^2\right)  g - \left( g_{xx} + \left(\frac c\nu - \frac b\nu \frac{f(\hat{v})}{\hat{v}_x} \right)g_x\right) 
= \left( \kappa + \left( \frac c{2\nu}\right)^2\right) e^{\frac c{2\nu} x} \, . 
\end{equation} 
Moreover, $|g_x (x)|\le \frac c{2\nu} g(x)$ for all $x\in\R$ and we have the lower bound 
$\int g^2 w^2\, dx \ge\frac{\left( \int e^{-\frac c{2\nu}x} \hat{v}_x^2\, dx\right)^2}{\int e^{-\frac c\nu x}\hat{v}_x^2\, dx}$.  
\end{lemma} 

\begin{proof} 
Fix a 1D-Brownian motion $(W_t)_{t\ge 0}$ defined on some underlying probability space $(\Omega , \cA , P)$. For all initial 
conditions $x\in \R$ let $X_t (x)$ be the unique strong solution of the stochastic differential equation 
\begin{equation} 
\label{lem2_1a_eq2} 
dX_t (x) = \left(\frac c\nu - \frac b\nu \frac{f(\hat{v})}{\hat{v}_x}(X_t (x))\right)\, dt + dW_t \, , X_0 (x) = x\, . 
\end{equation}  
The family of solutions is a Markov process on $\R$ having invariant measure $w^2\, dx$, i.e., 
$$ 
\int_{\R} E\left( h(X_t (h) \right) w^2\, dx = \int_{\R} h\, w^2\, dx\, , t\ge 0\, . 
$$ 
It follows that the associated semigroup of transition operators $p_t h(x) := E\left( h(X_t (x))\right)$ induces a contraction 
semigroup of Markovian integral operators on $L^p (\R , w^2\, dx)$ for all $p\in [1, \infty ]$. 

\medskip 
\noindent 
Theorem V.7.4 in \cite{Kr95} yields that the function 
$$ 
g(x) := \left( \kappa + \left( \frac c{2\nu}\right)^2\right)  \int_0^\infty e^{-\left( \kappa + \left( \frac c{2\nu}\right)^2\right) t } E\left( e^{\frac c{2\nu} X_t (x)} \right) \, dt
$$ 
is twice continuously differentiable and solves equation \eqref{lem1a_eq1}. Since 
$e^{\frac c{2\nu}x}\in L^2 (\R, w^2\, dx )$ we also have that $g\in L^2 (\R, w^2\, dx)$. 

\medskip 
\noindent 
We will show next the pointwise estimate of the derivative $g_x$. The solution $X_t (x)$ of the stochastic differential 
equation \eqref{lem2_1a_eq2} is differentiable w.r.t. its initial condition $x$. Its differential $DX_t (x)$ is the solution of the linear linear differential equation 
$$ 
dDX_t (x) = - 2\frac d{dx} \frac b\nu \frac{f(\hat{v})}{\hat{v}_x} \left( X_t (x) \right) DX_t (x)\, dt \, , DX_0 (x) = 1\, , 
$$ 
with explicit solution 
$$ 
DX_t (x) = \exp \left( - 2\frac b\nu \int_0^t \frac d{dx} \frac{f(\hat{v})}{\hat{v}_x}  \left( X_s (x) \right)\, ds \right) < 1
$$ 
for all $t > 0$, since $\frac d{dx} \frac b\nu \frac{f(\hat{v})}{\hat{v}_x} > 0$ according to Proposition \ref{prop0}. Consequently, 
$$ 
\begin{aligned} 
g_x (x)  
& = \frac c{2\nu}\left( \kappa + \left( \frac c{2\nu}\right)^2\right)   \int_0^\infty  e^{- \left( \kappa + \left( \frac c{2\nu}\right)^2\right) t} E\left( e^{\frac c{2\nu} X_t (x)} DX_t (x) \right) \, dt \\ 
& = \frac c{2\nu}\left( \kappa + \left( \frac c{2\nu}\right)^2\right)   \int_0^\infty  e^{- \left( \kappa + \left( \frac c{2\nu}\right)^2\right) t} E\left( e^{\frac c{2\nu} X_t (x) 
- 2\frac b\nu \int_0^t \frac d{dx} \frac b\nu \frac{f(\hat{v})}{\hat{v}_x} \left( X_s (x) \right)\, ds}\right) \, dt 
\end{aligned} 
$$ 
which implies that 
$$
|g_x (x)| < \frac c{2\nu} \left( \kappa + \left( \frac c{2\nu}\right)^2\right)  \int_0^\infty  e^{- \left( \kappa + \left( \frac c{2\nu}\right)^2\right)t } E\left( e^{\frac c{2\nu} X_t (x)} \right)\, dt 
= \frac c{2\nu} g(x) \, . 
$$ 

\medskip 
\noindent 
It remains to prove the lower bound. To this end note that invariance of the measure $w^2\, dx$ implies 
$$ 
\begin{aligned} 
\int g^2 w^2\, dx 
& \ge \left( \int w^2\, dx \right)^{-1} \left( \int g w^2\, dx \right)^2 \\ 
& =  \left( \int w^2\, dx \right)^{-1} \left( \left( \kappa + \left( \frac c{2\nu}\right)^2\right) \int_0^\infty e^{-\left( \kappa + \left( \frac c{2\nu}\right)^2\right) t} 
\int p_t \left( e^{\frac c{2\nu} x} \right) w^2\, dx\, dt\, \right)^2 \\ 
& = \left( \int w^2\, dx \right)^{-1} \left(\left( \kappa + \left( \frac c{2\nu}\right)^2\right) \int_0^\infty e^{-\left( \kappa + \left( \frac c{2\nu}\right)^2\right) t} \, dt  \int e^{\frac c{2\nu} x} w^2\, dx\, \right)^2 \\ 
& = \left( \int w^2\, dx \right)^{-1}\left( \int e^{\frac c{2\nu} x} w^2\, dx\, \right)^2   \, .  
\end{aligned} 
$$ 
\end{proof}

\bigskip
\begin{proof} (of Proposition \ref{prop2_1a}).   
Let $\tilde{h} := h - \frac{h(x_1)}{g(x_1)} g$, hence $\tilde{h}(x_1) = 0$. 
Then Proposition \ref{PropHardy} implies that 
$$ 
\int \tilde{h}^2\, w^2\, dx \le \frac 1{\kappa + \left( \frac c{2\nu}\right)^2} \int \tilde{h}_x^2\, w^2\, dx 
$$ 
or equivalently, 
$$ 
\int h^2 w^2 \, dx \le \frac 1{\kappa + \left( \frac c{2\nu}\right)^2} \int h_x^2\, w^2\, dx + T (h) 
$$ 
with the remainder 
$$ 
\begin{aligned} 
T(h) 
& := \frac 1{\kappa + \left( \frac c{2\nu}\right)^2} \left( -2 \int h_x  \frac{h(x_1)}{g(x_1)} g_x w^2 dx  + 
\left( \frac{h(x_1)}{g(x_1)} \right)^2 \int g_x^2 w^2\, dx \right) \\
& \quad + 2\int h\frac{h(x_1)}{g(x_1)} g w^2 dx  - \left( \frac{h(x_1)}{g(x_1)} \right)^2 \int g w^2\, dx  
\end{aligned} 
$$ 
Using Lemma \ref{lem1a} we obtain that 
$$ 
\begin{aligned} 
T(h) 
& = 2 \frac{h(x_1)}{g(x_1)} \int \left(  g -\frac 1{\kappa + \left( \frac c{2\nu}\right)^2} \left( g_{xx} - \left(\frac c\nu - 2\frac b\nu \frac{f(\hat{v})}{\hat{v}_x} \right) g_x
\right) \right) h\, w^2\, dx \\ 
& \qquad + \left( \frac{h(x_1)}{g(x_1)} \right)^2 \left( \frac 1{\kappa + \left( \frac c{2\nu}\right)^2} \int g_x^2\, w^2\, dx - \int g^2\, w^2\, dx \right) \\ 
& \le  2\frac{h(x_1)}{g(x_1)} \int e^{\frac c{2\nu} x} h\, w^2\, dx 
- \left( \frac{h(x_1)}{g(x_1)} \right)^2 \frac{\kappa}{\kappa + \left( \frac c{2\nu}\right)^2} \int g^2 \, w^2\, dx \, . 
\end{aligned} 
$$ 
In the last inequality we have used the pointwise estimate $|g_x(x)|\le \frac c{2\nu} g(x)$. Using the lower bound 
$\int g^2 w^2\, dx \ge   \frac{\left( \int e^{-\frac c{2\nu}x} \hat{v}_x^2\, dx\right)^2}{\int e^{-\frac c\nu x}\hat{v}_x^2\, dx}$ obtained in the previous Lemma we conclude that 
$$ 
\begin{aligned} 
T(h) & \le \frac{\kappa + \left( \frac c{2\nu}\right)^2}{\kappa} \left( \int g^2 w^2\, dx \right)^{-1} \left( \int h e^{\frac c{2\nu}x}w^2\, dx\right)^2  \\ 
& \le C_{\ref{prop2_1a}}  \left( \int h e^{\frac c{2\nu}x}w^2\, dx\right)^2 
\end{aligned} 
$$ 
with 
$$ 
C_{\ref{prop2_1a}} = \frac {\kappa + \left( \frac c{2\nu}\right)^2}\kappa 
\frac{\int e^{-\frac c\nu x}\hat{v}_x^2\, dx}{\left( \int e^{-\frac c{2\nu}x} \hat{v}_x^2\, dx\right)^2} 
$$ 
which implies the assertion. 
\end{proof} 

\medskip 
\noindent 
Having Proposition \ref{prop2_1a} we can now state the following 

\begin{proposition} 
\label{prop2_1} 
Let $u\in C_c^1 (\R)$ and write $u = hw$ for $h\in C_c^1 (\R)$. Then 
$$ 
\begin{aligned} 
\langle \nu \Delta u + b f^\prime (\hat{v} ) u,u\rangle 
& \le -\nu \frac{\kappa}{\kappa + \left( \frac c{2\nu}\right)^2}  \int h_x^2 w^2\, dx \\
& \qquad + \nu\left(\frac c{2\nu}\right)^2 C_{\ref{prop2_1a}} \left( \int u \hat{v}_x\, dx\right)^2 \, . 
\end{aligned} 
$$
\end{proposition} 

\begin{proof} 
First note that $h\in C_c^1 (\R)$, and thus equations \eqref{eq2_1} and Proposition \ref{prop2_1a} imply that 
$$ 
\begin{aligned} 
& \langle \nu \Delta u + b f^\prime (\hat{v} ) u,u\rangle 
  = - \nu  \int h^2_x w^2 \, dx + \nu\left(\frac{\nu}{2c} \right)^2  \int h^2 w^2 dx \\ 
& \qquad \le - \nu \frac{\kappa}{\kappa + \left( \frac c{2\nu}\right)^2} \int h_x^2\, w^2\, dx  
         + \nu\left(\frac c{2\nu}\right)^2  C_{\ref{prop2_1a}} \left( \int \tilde{h} e^{\frac c{2\nu}x} w^2\, dx\right)^2 \\ 
& \qquad = - \nu \frac{\kappa}{\kappa + \left( \frac c{2\nu}\right)^2} \int h_x^2\, w^2\, dx  
         + \nu\left(\frac c{2\nu}\right)^2  C_{\ref{prop2_1a}} \left( \int u \hat{v}_x \, dx\right)^2 \, . 
\end{aligned} 
$$
\end{proof}

\noindent 
In the next step we will show that for $u\in C_c^1 (\R)$ and $u = h\hat{v}_x$ its $V$-norm 
$\|u\|_V$ can be controlled by $\int \left( he^{\frac c{2\nu} x}\right)^2_x w^2\, dx$. 

\begin{lemma}
\label{lem2_1} 
Let $u\in C_c^1 (\R )$ and write $ u = hw$. Then 
$$ 
\|u\|_V^2 \le q_1 \int h^2_x w^2\, dx + q_2 \langle u , \hat{v}_x\rangle^2 
$$ 
where 
$$ 
q_1 := \left( 1 + \left( \frac {b\eta}{\nu} + 1\right) \frac 1{\kappa +  \left( \frac c{2\nu}\right)^2} \right)\, , 
\qquad 
q_2 := \left( \frac{b\eta}{\nu} + 1\right)  C_{\ref{prop2_1a}} \, , 
$$ 
and 
$$
\eta := \max_{v\in [0,1]} f^\prime (v)\, . 
$$ 
\end{lemma} 

\begin{proof} 
Using \eqref{eq2_1} we have that 
$$ 
\begin{aligned} 
\nu \int u_x^2\, dx 
& = - \langle \nu\Delta u + b f^\prime (\hat{v} ) u , u\rangle + b\langle f^\prime (\hat{v}) u , u \rangle \\ 
& \le \nu \int h^2_x w^2\, dx + b\eta \|u\|_H^2 \, . 
\end{aligned}  
$$
Proposition \ref{prop2_1a} now implies  
$$ 
\begin{aligned} 
\|u\|_V^2 
& \le \left( 1 + \left( \frac {b\eta}{\nu} + 1\right) \frac 1{\kappa +  \left( \frac c{2\nu}\right)^2} \right) \int h^2_x w^2\, dx  \\
& \qquad    + \left( \frac{b\eta}{\nu} + 1\right) C_{\ref{prop2_1a}}  \langle u,\hat{v}_x\rangle^2 \, , 
\end{aligned} 
$$ 
which implies the assertion. 
\end{proof}

\begin{proof} (of Theorem \ref{th1}) First let $u\in C_c^1 (\R )$. Then Proposition \ref{prop2_1} implies the estimate 
$$ 
\begin{aligned} 
\langle \nu \Delta u + bf^\prime (\hat{v}) u , u \rangle 
& \le - \nu \frac{\kappa}{\kappa + \left( \frac c{2\nu}\right)^2} \int h_x^2\, w^2\, dx \\ 
& \qquad + \nu \left(\frac c{2\nu}\right)^2 C_{\ref{prop2_1a}}\langle u , \hat{v}_x\rangle^2 \, . 
\end{aligned} 
$$ 
Combining the last estimate with the previous Lemma \ref{lem2_1}, we obtain that 
$$ 
\begin{aligned} 
\langle \nu \Delta u + b f^\prime (\hat{v})u,u\rangle 
& \le -  \frac{\kappa}{\kappa + \left( \frac c{2\nu}\right)^2} \frac {\nu}{q_1} \|u\|_V^2 \\
& \qquad   + \left(\frac\kappa{\kappa +  \left( \frac c{2\nu}\right)^2} \frac{\nu q_2}{q_1} 
+ \nu \left(\frac c{2\nu}\right)^2 C_{\ref{prop2_1a}}  \right) \langle u,\hat{v}_x\rangle^2\, 
\end{aligned} 
$$ 
which implies Theorem \ref{th1} with 
$$ 
\kappa_\ast = \frac{\kappa}{\kappa + \left( \frac c{2\nu}\right)^2} \frac {\nu}{q_1}
$$ 
and 
$$ 
C_\ast = \left( \kappa_\ast q_2 + \frac{\nu}{\kappa} \left(\frac c{2\nu}\right)^2 \left( \kappa + \left(\frac c{2\nu}\right)^2\right)  
\frac{\int e^{-\frac c\nu x }\hat{v}_x^2\, dx}{\left(\int e^{-\frac c{2\nu}x }\hat{v}_x^2\, dx \right)^2}  \right)\, . 
$$ 
\end{proof} 

\medskip 
\noindent 
It remains to prove the weighted Hardy type inequality \eqref{WeightedHardy} which is of independent interest. 

\begin{proposition} 
\label{PropHardy} 
Let $w\in C_b^2 (\R)$ and $\theta = \frac{w_x}{w}$. Suppose that  
$$ 
\inf_{x\in\R} - \theta ' (x) + \theta^2 (x) \ge \kappa_0 > 0 
$$ 
and that there exists $\hat{x}$ such that $\theta (\hat{x}) = 0$. Then 
$$ 
\int h^2\, w^2\, dx \le \frac 1\kappa \int h_x^2\, w^2\, dx 
$$ 
for any $h\in C_b^1 (\R)$ with $h(\hat{x}) = 0$. 
\end{proposition} 

\begin{proof} 
Define the function $g(x) := \left(-\theta ' (x) + \theta^2 (x)\right)\exp \left( - \int_{\hat{x}}^x \theta (s)\, ds \right)$ and notice that 
$$ 
\exp \left( - \int_{\hat{x}}^x \theta (s)\, ds \right) = \exp \left( - \log w (x) + \log w (\hat{x})\right)  
= \frac{w(\hat{x})}{w(x)}  
$$
and thus 
$$ 
g(x) = \left( - \theta ' (x) + \theta ^2 (x)\right) \frac{w(\hat{x})}{w(x)} \ge \kappa \frac{w(\hat{x})}{w(x)}\, . 
$$ 
Then for $x\ge \hat{x}$ we have that 
$$ 
\begin{aligned} 
\left( h(x) - h(\hat{x} )\right)^2 
& = \left( \int_{\hat{x}}^x h_x (s)\, ds \right)^2 
  \le \int_{\hat{x}}^x \frac 1{g(s)} h_x^2 (s)\, ds \int_{\hat{x}}^x g(s)\, ds \\ 
& = \int_{\hat{x}}^x \frac 1{g(s)} h_x^2 (s) \, ds \left( - \theta (x)\exp\left( -\int_{\hat{x}}^x \theta (s)\, ds\right) \right)  \\
& =  \int_{\hat{x}}^x \frac 1{g(s)} h_x^2 (s) \, ds \left( - \frac{w_x (x)}{w(x)} \frac{w (\hat{x} )}{w(x)}\right) \\ 
& \le \frac 1\kappa  \int_{\hat{x}}^x \frac{w(s)}{w(\hat{x})}  h_x^2 (s) \, ds \left( - \frac{w_x (x)}{w(x)} \frac{w (\hat{x} )}{w(x)}\right) \, .   
\end{aligned} 
$$ 
Integrating against $w^2\, dx$ for $x\ge \hat{x}$ now yields the following estimate 
\begin{equation} 
\label{eq_1_2_1} 
\begin{aligned} 
\int_{\hat{x}}^\infty \left( h - h(\hat{x} )\right)^2  w^2 \, dx 
& \le \frac 1\kappa \int_{\hat{x}}^\infty  \frac{w(s)}{w(\hat{x})} h_x^2 (s) \int_s^\infty - w_x (x) w(\hat{x})\, dx \, ds  \\
& =  \frac 1\kappa \int_{\hat{x}}^\infty h_x^2 (s) w^2 (s) \, ds  \, . 
\end{aligned} 
\end{equation} 
Similarly, for $x\le \hat{x}$ we have that 
$$ 
\begin{aligned} 
\left( h(\hat{x}) - h(x)\right)^2 
& = \left( \int_x^{\hat{x}} h_x (s)\, ds \right)^2 
  \le \int_x^{\hat{x}} \frac 1{g(s)} h_x^2 (s)\, ds \int_x^{\hat{x}} g(s)\, ds \\ 
& = \int_x^{\hat{x}} \frac 1{g(s)} h_x^2 (s) \, ds \left( \theta (x) \exp \left(- \int_{\hat{x}}^x \theta (s)\, ds\right) \right)  \\
& =  \int_x^{\hat{x}} \frac 1{g(s)} h_x^2 (s) \, ds \frac{w_x (x)}{w(x)} \frac{w (\hat{x} )}{w(x)} \\ 
& \le \frac 1\kappa  \int_x^{\hat{x}} \frac{w(s)}{w(\hat{x})}  h_x^2 (s) \, ds \frac{w_x (x)}{w(x)} \frac{w (\hat{x} )}{w(x)} \, .   
\end{aligned} 
$$ 
Integrating against $w^2\, dx$ now for $x\le \hat{x}$ yields 
\begin{equation} 
\label{eq_1_2_2} 
\begin{aligned} 
\int_{-\infty}^{\hat{x}} \left( h-h(\hat{x} )\right)^2 & w^2\, dx  \\  
& \le \frac 1\kappa \int_{-\infty}^{\hat{x}}  \frac{w(s)}{w(\hat{x})} h_x^2 (s) \int_{-\infty}^{\hat{x}} w_x (x) w(\hat{x})\, dx \, ds  \\
& =  \frac 1\kappa \int_{-\infty}^{\hat{x}} h_x^2 (s) w^2 (s) \, ds  \, . 
\end{aligned} 
\end{equation}  
The assertion now follows from estimates \eqref{eq_1_2_1} and \eqref{eq_1_2_2}. 
\end{proof}

\bigskip 
\noindent 
{\bf Acknowlegdement} This work is supported by the BMBF, FKZ 01GQ1001B.


\begin{thebibliography}{99}
\index{References}


\bibitem{ET}
Ermentrout, G.B., Terman, D.H., 
Mathematical Foundations of Neuroscience, 
Springer, Berlin, 2010. 


\bibitem{Ev} 
Evans, J.W.,  
Nerve axon equation III: Stability of the nerve impulse, 
Indiana Univ. Mat. J., Vol. 22, 577--594, 1972.  


\bibitem{FMcL} 
Fife, P.C., McLeod, J.B., 
The approach of solutions of nonlinear diffusion equations to travelling front solutions, 
Arch. Ration. Mech. Anal., Vol. 65, 335--361, 1977. 


\bibitem{HR}
Hadeler, K.P., Rothe, F., 
Travelling Fronts in Nonlinear Diffusion Equations, 
J. Math. Biol., Vol. 2, 251--263, 1975.  

\bibitem{Henry}
Henry, D., 
Geometric theory of semilinear parabolic equations, 
LNM Vol. 840, Springer-Verlag, Berlin, 1981. 

\bibitem{Jones}
Jones, C.K.R.T., 
Stability of the traveling wave solution of the FitzHugh-Nagumo equations, 
Trans A.M.S., Vol. 286, 431--469, 1984. 



\bibitem{Kr95} 
Krylov, N.V., 
Introduction to the Theory of Diffusion Processes, 
American Mathematical Society, Providence, RI, 1995. 


\bibitem{LR}
Liu, W., R\"ockner, M.,  
SPDE in Hilbert space with locally monotone coefficients, 
J. Funct. Anal., Vol. 295, 2902--2922, 2010. 

\bibitem{LT}
Lord, G.J., Th\"ummler, V.,  
Computing Stochastic Travelling Waves, 
SIAM Journal of Scientific Computation, Vol. 34, 24--43, 2012.


\bibitem{OR}
Otto, F., Reznikoff, M.G., 
Slow motion of gradient flows, 
J. Differential Equations, Vol. 237, 372--420, 2007.


\bibitem{PR}
Prevot, C., R\"ockner, M., 
A Concise course on Stochastic Partial Differential Equations, 
Lecture Notes in Mathematics, Vol. 1905, Springer, Berlin, 2007.

\bibitem{St}
Stannat, W., 
Stability of travelling waves in stochastic Nagumo equations, 
arXiv:1301.6378, 2013. 

\bibitem{T2010}
Tuckwell, H.C., Jost, J., 
Weak noise in neurons may powerfully inhibit the generation of 
repetitive spiking but not its propagation, 
PLoS Comput. Biol., Vol. 6, 13 pp, 2010. 

\bibitem{T2011} 
Tuckwell, H.C., Jost, J.,
The effect of various spatial distributions of weak noise on rhythmic spiking, 
J. Comput. Neurosci., Vol 30, 361--371, 2011.


\end{thebibliography}
\end{document}